%% file: ex_article.tex
\documentclass[onefignum,onetabnum]{siamart251216}

\input{ex_shared}

\newif\ifappendsupplement
\appendsupplementtrue

\ifappendsupplement
\else
\fi

\ifpdf
\hypersetup{
  pdftitle={Property-Preserving Data Compression for Incompressible Flows},
  pdfauthor={Andrey Pluzhnik, Ramansh Sharma, Martin Burtscher, and Varun Shankar}
}
\fi

\begin{document}

\maketitle

\begin{abstract}
We develop a method for lossy compression of incompressible (i.e., divergence-free) velocity
fields that preserves zero divergence and, on periodic domains,
periodicity. We fit the nodal velocities with property-preserving
matrix-valued kernels and compress their coefficients using standard
lossy compressors. Every decompressed coefficient vector defines a field
with the same constraints, independently of the compression tolerance. We present several fitting schemes designed around either low-rank representation of or sparsity in the interpolation matrix. More specifically, we show results for global Mat\'ern interpolation, a two-level \nyst{} approximation,
\nyst{}--Wendland residual interpolation, multilevel Wendland
interpolation and a property-preserving (redundant) kernel frame. We derive nodal error bounds
that relate the compressor tolerance to the additional velocity error. Experiments with four compressors on cavity,
Taylor--Green, and species-transport flows verify the bounds under their
coefficient-error assumptions and compare storage at matched velocity
accuracy. The low-rank methods give compact representations when they resolve
the velocity field sufficiently accurately; residual, multilevel, and frame interpolation
extend the attainable accuracy when they do not.
We also measure changes in turbulent velocity
statistics. The resulting representations support an accuracy--compression
tradeoff while preserving incompressibility and periodicity analytically.
\end{abstract}



\input{introduction}
\input{methods}
\input{guarantees}
\input{results}

\input{conclusion}

\section*{Acknowledgments}
AP, RS, and VS were supported by the US National Science Foundation (NSF) grant SHF 2403379. RS and VS were also supported by NSF grant DMS 2505986. MB was supported by NSF grant SHF 2403380.

\section*{Declaration of AI assistance}
The authors used OpenAI's Codex to assist with manuscript editing, code development, and computational experiment workflows.
The authors assume responsibility for all content.

\bibliographystyle{siamplain}
\bibliography{references}

\ifappendsupplement
  \clearpage
  \begingroup
    \makeatletter
    \def\siamprelabel{SM}
    \makeatother
      \let\SMcite\cite
      \renewcommand{\cite}[1]{{\def\siamprelabel{}\SMcite{#1}}}

    \setcounter{page}{0}
    \setcounter{section}{0}
    \setcounter{subsection}{0}
    \setcounter{subsubsection}{0}
    \setcounter{equation}{0}
    \setcounter{figure}{0}
    \setcounter{table}{0}

    \renewcommand{\theHsection}{SM.\arabic{section}}
    \renewcommand{\theHsubsection}{SM.\arabic{section}.\arabic{subsection}}
    \renewcommand{\theHsubsubsection}
      {SM.\arabic{section}.\arabic{subsection}.\arabic{subsubsection}}
    \renewcommand{\theHfigure}{SM.\arabic{figure}}
    \renewcommand{\theHtable}{SM.\arabic{table}}
    \renewcommand{\theHequation}{SM.\arabic{equation}}

    \input{ex_supplement_body}
  \endgroup
\fi

\end{document}

\end{document}

%% file: ex_shared.tex
\usepackage{amsfonts}
\usepackage{amsmath,amssymb}
\usepackage{booktabs}
\usepackage{array}
\usepackage{needspace}
\usepackage{graphicx}
\usepackage{microtype}
\usepackage[nocompress]{cite}
\usepackage[section]{placeins}
\usepackage{float}

\ifpdf
  \DeclareGraphicsExtensions{.pdf,.png,.jpg,.eps}
\else
  \DeclareGraphicsExtensions{.eps}
\fi

\newcommand{\ddo}{\"{o}}
\newcommand{\nyst}{Nystr\ddo{}m}

\newsiamremark{remark}{Remark}
\newsiamremark{hypothesis}{Hypothesis}
\crefname{hypothesis}{Hypothesis}{Hypotheses}
\crefname{section}{Section}{Sections}
\crefname{subsection}{Section}{Sections}
\crefname{subsubsection}{Section}{Sections}
\Crefname{section}{Section}{Sections}
\Crefname{subsection}{Section}{Sections}
\Crefname{subsubsection}{Section}{Sections}
\newsiamthm{claim}{Claim}

\headers{A. Pluzhnik, R. Sharma, M. Burtscher, and V. Shankar}
{Property-Preserving Data Compression}

\title{Property-Preserving Data Compression for Incompressible Flows}

\author{Andrey Pluzhnik\thanks{Kahlert School of Computing, University of
Utah, Salt Lake City, UT 84112, USA.}
\and Ramansh Sharma\footnotemark[1]
\and Martin Burtscher\thanks{Department of Computer Science, Texas State
University, San Marcos, TX 78666, USA.}
\and Varun Shankar\footnotemark[1]\thanks{Corresponding author:
\email{shankar@cs.utah.edu}.}}

\usepackage{amsopn}


%% file: introduction.tex
\section{Introduction}
\label{sec:introduction}

High-fidelity flow simulations produce large collections of velocity
fields whose storage and transfer can be costly. Lossy compression
reduces these costs by permitting controlled errors in the stored values
\cite{Lindstrom2014,LiangEtAl2023,GongEtAl2024,FallinEtAl2025}.
For scientific use, the reconstructed fields must also retain the
properties needed for subsequent computation and analysis
\cite{DiEtAl2025Survey,CappelloEtAl2025Requirements}.
In incompressible flow,\footnote{Incompressibility refers to the fluid
constraint $\nabla\cdot u=0$. Data compression refers to reducing the
storage required for its numerical representation.}
these include zero divergence and, on periodic domains, agreement across
identified periodic boundaries. Small velocity errors alone do not ensure either
property: differentiation can amplify compression errors, and independently
compressed boundary values can disagree. Compression can also change
velocity fluctuations across spatial scales, affecting statistics used
to characterize turbulent flow.

Existing compressors preserve vector-field critical
points and separatrices~\cite{TheiselEtAl2003,XiaEtAl2024}, control selected
quantities of interest~\cite{JiaoEtAl2022,LeeEtAl2022}, or penalize divergence
and velocity-gradient errors during training
\cite{GlawsKingSprague2020,MomenifarEtAl2022}. A Helmholtz--Hodge
decomposition can extract a divergence-free field from perturbed velocity
data~\cite{FuselierWright2017}. Applied after decompression, this requires
an additional solve, and the resulting field error must account for the
projection. We instead fit the sampled velocities as a linear combination
of divergence-free functions and compress the weights in that combination.
Decompression changes the weights, but every resulting linear combination
remains divergence-free. The same fitted representation can therefore be
paired with different compressors without changing this guarantee.

Matrix-valued kernels constructed from radial basis functions (RBFs) offer a convenient starting point and can be used on regular grids, simplicial meshes, and point clouds. For each center, a kernel supplies functions describing how its
vector coefficient contributes to the velocity throughout the domain.
Differentiating a scalar positive-definite kernel yields functions with
zero divergence or zero curl~\cite{NarcowichWard1994,Lowitzsch2005,Wendland2005}. These kernels
have been used for Stokes flow~\cite{Wendland2009},
Helmholtz--Hodge decomposition~\cite{FuselierWright2017}, localized
approximation~\cite{DrakeFuselierWright2021}, magnetohydrodynamics, and ocean
flow reconstruction~\cite{McNally2011,VennellBeatson2009}, and for the surrogate modeling of incompressible flows~\cite{SharmaEtAl2026Fluids}. We use these kernels to preserve constraints when
lossy compression perturbs the coefficients of a fitted field.

For periodic flow, we construct functions that repeat after a domain-length
shift in either coordinate. We do this by mapping each periodic coordinate
onto a circle before evaluating the scalar kernel, then applying the
divergence-free construction~\cite{LightCheney1992,FuselierWright2012,SharmaEtAl2026Fluids}.
For flow with variation over many spatial scales, we combine kernels with
different length scales, following the turbulence-motivated construction
in~\cite{SharmaEtAl2026Fluids}. These mixtures still enforce zero divergence.
Their fitted coefficients determine the velocity statistics, which we
compare before and after compression.

We use three flows to examine these different requirements. Lid-driven
cavity flow tests incompressibility in the presence of localized wall
layers. Taylor--Green flow adds periodicity and has smooth spatial
structure that may admit a compact representation. The species-transport
velocity data contain variations over many spatial scales, allowing us
to examine both incompressibility and changes in velocity statistics.
Together, these examples test whether a small approximation space suffices
and what additional storage is needed when it does not.

Within these property-preserving spaces, we must choose a representation
that is both accurate and economical to store. Interpolation with globally
supported Mat\'ern kernels
provides a simple and accurate starting point: it fits every nodal velocity using one vector
coefficient per node. To reduce this coefficient count, we use a \nyst{}
approximation with matrix-valued kernels centered at a subset of nodes,
called landmarks~\cite{WilliamsSeeger2001}. To improve the accuracy of this
reduced fit, we develop a two-level \nyst{} procedure that fits a second set of kernel functions to the
residual (the difference between the data and the first fit). This can reduce storage,
but any error left outside the two spaces remains even with lossless
coefficient storage. To remove this remaining error, we also develop a \nyst{} + Wendland
Hybrid that interpolates the residual with compactly supported Wendland
functions~\cite{Wendland1995,Wendland2005} centered at every node, so
that the uncompressed representation interpolates the data. 

Multilevel Wendland interpolation offers another way to reduce the cost of kernel fitting:
successive kernel corrections approximate the remaining residual on nested
center sets~\cite{Wendland2010,FarrellGillowWendland2017}. We use a final
all-node level to recover the nodal data. Local supports can make these
fitting systems sparse. As an alternative to sequential residual fitting,
we also develop a property-preserving extension of the Frame Kernel
Method~\cite{FriedenEtAl2026}. It combines several sets of kernel functions
and fits their coefficients simultaneously. With more coefficients than
nodal conditions, the frame interpolates the data using a minimum-norm
coefficient vector.

Our five representations therefore compare a single global fit,
successive residual fits, and a simultaneous redundant fit. Their
compressed sizes must be measured at matched velocity accuracy; neither
coefficient count nor coefficient norm alone determines storage. We pair each representation with ZFP, Portable Floating-Point Lossy (PFPL),
SZ3, and MultiGrid Adaptive Reduction of Data (MGARD), and compare it
with direct nodal compression. We also provide a theorem that bounds the additional nodal error using the compressor's coefficient
tolerance and the evaluation matrix. Reported compression ratios count
storage for each velocity field; node coordinates, center indices, and
kernel parameters reused across fields are excluded.
We also use NVIDIA's CUDA platform for graphics processing units (GPUs) to compress and
decompress the coefficient arrays without transferring them to host
memory. These experiments compare central processing unit (CPU) and GPU compression at matched
field errors, accounting for their different compression modes.

For the Taylor--Green example, a compact periodic landmark space
gives substantially higher compression than direct nodal storage. For the
tested cavity and species-transport landmark sets, a full-node correction
is needed to attain small nodal errors within the reduced representation.
Among the interpolating kernel
representations, the best compression depends on the flow and error limit.
We present the kernel representations in \cref{sec:methods}, prove the nodal
error bounds in \cref{sec:guarantees}, and compare accuracy, compression,
and property changes in \cref{sec:results}. We present conclusions in \cref{sec:conclusions}. Experimental parameters appear
in \cref{sec:experiments} of the supplementary materials. The prefix SM
identifies supplementary sections, equations, figures, and tables.

%% file: methods.tex
\section{Property-preserving compression methods}
\label{sec:methods}

Let $X=\{x_i\}_{i=1}^N\subset\Omega\subset\mathbb R^d$, with $d=2$ or $3$,
be distinct sampling sites. We collect their velocities in
$u_X=(u(x_1)^T,\ldots,u(x_N)^T)^T\in\mathbb R^{dN}$.
We seek a representation whose reconstructed field remains divergence-free
after lossy compression. Taylor--Green additionally requires periodicity,
while the turbulent species-transport example motivates kernels with several physical
length scales. We first construct these kernels and describe coefficient
compression, then develop the five representations introduced in
\cref{sec:introduction}. Each defines a continuous velocity field that we
evaluate at the original nodes. Interpolants match the supplied nodal
values before compression; reduced fits may already have a nonzero error.

\subsection{Kernel representations and compression}
\label{sec:kernel-compression}

The kernel determines which field properties hold for every coefficient
vector. The fitting rule then determines the nodal accuracy and the
coefficients available for compression. We separate these two choices so
that all five representations can use the same property-preserving
construction.

\subsubsection{Property-preserving kernels}
\label{sec:property-kernels}

\paragraph{Divergence-free functions}
\label{sec:matrix-kernels}

We use globally supported Mat\'ern kernels~\cite{Wendland2005} for the global and landmark
fits, and compactly supported Wendland kernels~\cite{Wendland1995,Wendland2005} for residual corrections
and frames. Wendland kernels are piecewise polynomials that vanish beyond
a prescribed distance, thereby providing compact support and sparsity. Both scalar functions are $C^4 (\mathbb{R}^d)$, so the differentiations below
leave smooth matrix-valued kernels. For
the dimensionless distance $t\geq0$, they are
\begin{align}
 \phi_M(t) &= e^{-t}(t^2+3t+3),
 \label{eq:matern-c4}\\
 \phi_W(t) &= (1-t)_+^6(35t^2+18t+3).
 \label{eq:wendland-c4}
\end{align}
Here $(a)_+=\max(a,0)$, so $\phi_W(t)$ vanishes for $t\geq1$.
The Mat\'ern normalization differs from the usual one only by a positive
constant. The Wendland function is positive definite in dimensions up to
three~\cite{Wendland1995,Wendland2005}; the periodic construction below
requires a function valid in dimension four. A shape parameter
$\varepsilon>0$ sets the spatial length scale through
$\varphi_\varepsilon(x,y)=\phi(\varepsilon\|x-y\|_2)$.
For the Wendland kernel, the support radius is $1/\varepsilon$.

To obtain vector-valued basis functions with zero divergence, we
differentiate the scalar kernel in both arguments~\cite{NarcowichWard1994,Lowitzsch2005}. Let $\mathcal C_x$
denote the curl acting on the $x$ variable. The resulting matrix-valued
kernel is a $d\times d$ matrix for each pair of spatial points:
\begin{equation}
 \Phi_{\mathrm{div}}(x,y)
 =\mathcal C_x\mathcal C_y^T\varphi_\varepsilon(x,y)
 =\left(-\Delta_x I+\nabla_x\nabla_x^T\right)
   \varphi_\varepsilon(x,y),
 \label{eq:div-kernel}
\end{equation}
where the second equality holds for a stationary radial scalar kernel. In two
dimensions, $\mathcal C f=(\partial_{x_2}f,-\partial_{x_1}f)^T$; in three
dimensions it is the usual vector curl. Since divergence annihilates curl,
each column $\Phi_{\mathrm{div}}(\cdot,z)e_k$ is divergence-free, where
$e_k$ is the $k$th coordinate vector. Linear combinations of these columns
therefore give divergence-free fields:
\begin{equation}
 u_Z(x)=\sum_{j=1}^{M}\Phi_{\mathrm{div}}(x,z_j)c_j
 \quad\Longrightarrow\quad \nabla\cdot u_Z(x)=0
 \label{eq:kernel-expansion}
\end{equation}
for every center set $Z=\{z_j\}_{j=1}^M$ and all coefficients
$c_j\in\mathbb R^d$.
The kernel defines a reproducing kernel Hilbert space, also called its
native space. We denote it by $\mathcal H_\Phi$. The span of the functions
centered at $Z$ is a finite-dimensional subspace of $\mathcal H_\Phi$;
strict positive definiteness gives unique interpolation at distinct
sites~\cite{Lowitzsch2005,Wendland2009,FuselierWright2017}.

Write $\Phi$ for the chosen matrix-valued kernel, including the periodic
and multiscale variants constructed below. For evaluation nodes
$X=\{x_i\}_{i=1}^N$, define the block matrix
\begin{equation}
 K_\Phi(X,Z)=\bigl[\Phi(x_i,z_j)\bigr]_{i=1,j=1}^{N,M}
 \in\mathbb R^{dN\times dM}.
 \label{eq:block-kernel}
\end{equation}
For a kernel expansion $s$ with coefficients $c=(c_1^T,\ldots,c_M^T)^T$,
evaluation at the sampling nodes gives
\[
 s_X=K_\Phi(X,Z)c,
 \qquad
 s_X=(s(x_1)^T,\ldots,s(x_N)^T)^T.
\]
When $Z=X$,
interpolation with a strictly positive-definite kernel gives a square positive
definite block system. With a globally supported Mat\'ern kernel it is dense;
with a Wendland kernel entries vanish outside the support radius, although
broad supports can produce a nearly dense matrix. When $M<N$, the
evaluation matrix has $dN$ rows and $dM$ columns. Fitting the nodal values
then gives an overdetermined least-squares problem, or equivalently a
projection onto the sampled trial space. Combining several sets of kernel
functions can instead give more coefficients than nodal values. The
resulting interpolation problem is underdetermined and requires a rule
for selecting one coefficient vector. The representations below use
these three possibilities: square interpolation systems, overdetermined
projection problems, and underdetermined frame fits.

\paragraph{Periodic kernels}
\label{sec:periodic-multiscale-kernels}

The Taylor--Green domain identifies coordinates differing by $2\pi$ in
either direction. We map each coordinate onto a circle using its sine and
cosine. Their product is a two-dimensional torus in $\mathbb R^4$; points
separated by a period have the same image. The resulting embedding and
distance are
\begin{equation}
 h(x_1,x_2)=(\cos x_1,\sin x_1,\cos x_2,\sin x_2),
 \qquad
 r_\pi(x,y)=\|h(x)-h(y)\|_2.
 \label{eq:torus-embedding}
\end{equation}
The distance satisfies
$r_\pi^2=4-2\{\cos(x_1-y_1)+\cos(x_2-y_2)\}$ and is unchanged by a
periodic shift of either argument. The scalar kernel must be positive
definite in the embedding dimension. The Mat\'ern function in
\cref{eq:matern-c4} has this property. For Wendland, we use the $C^4$
function $\phi_{5,2}$, which is positive definite in dimensions up to
five~\cite{Wendland1995,Wendland2005}:
\begin{equation}
 \phi_{W,\pi}(t)=(1-t)_+^7(16t^2+7t+1).
 \label{eq:periodic-wendland-profile}
\end{equation}
Restricting either scalar kernel to the embedded torus preserves positive
definiteness~\cite{LightCheney1992,FuselierWright2012}. We then impose
zero divergence by differentiating with respect to the original spatial
coordinates, rather than the four embedding coordinates:
\begin{equation}
 \Phi_\pi(x,y)=\mathcal C_x\mathcal C_y^T
 \phi\!\left(\varepsilon r_\pi(x,y)\right).
 \label{eq:periodic-div-kernel}
\end{equation}
Every column of $\Phi_\pi$ is divergence-free and periodic throughout the domain
~\cite{SharmaEtAl2026Fluids}.
At distinct sites modulo the periods, the periodic kernel yields a symmetric
positive-definite interpolation matrix. It is the Gram matrix of first-derivative
evaluation functionals in the scalar kernel's native space. This space contains
smooth local functions whose derivatives can be prescribed independently at
distinct sites, so those functionals are linearly independent
\cite{Wendland2005,FuselierWright2012}. We identify duplicate periodic nodes
before forming the matrix and call the remaining nodes the active sites.

\paragraph{Physical-scale mixtures}

For species transport, we seek a kernel that represents velocity
variation across several physical scales. Second-order turbulence scaling
relates the mean squared velocity difference between two points to their
separation, with a $2/3$ power law in an inertial range. To incorporate this
scale dependence into the kernel, we combine terms with lengths $\ell_s$,
using the same construction for Mat\'ern and Wendland. Writing
$\eta=(\ell_1,\ldots,\ell_J)$ for the chosen lengths, we define
\begin{equation}
 \Phi_\eta(x,y)
 =\sum_{s=1}^{J}\ell_s^{8/3}
   \Phi_{\mathrm{div}}(x,y;\ell_s^{-1}),
 \qquad
 0<\ell_1<\cdots<\ell_J.
 \label{eq:power-law-kernel}
\end{equation}
The two derivatives in \eqref{eq:div-kernel} contribute a factor
$\ell_s^{-2}$. Multiplication by $\ell_s^{8/3}$ therefore gives each
matrix-kernel term an amplitude proportional to $\ell_s^{2/3}$, motivated
by second-order turbulence scaling~\cite{Kolmogorov1991,SharmaEtAl2026Fluids}.
Positive weights preserve positive definiteness, and every term is divergence-free.
As a covariance, the kernel describes correlations between velocities
at different points. The weights in \eqref{eq:power-law-kernel} incorporate
the desired scale dependence into those correlations. The statistics of
a fitted field also depend on its coefficients, so we measure their change
after compression in \cref{sec:property-results}.

The physical scales in the mixture and the approximation levels used below
serve different purposes. The mixture specifies the kernel covariance;
the levels provide successive or redundant sets of centers for fitting
that kernel to the data. We keep the lengths and weights fixed across
landmark counts and correction levels, with the complete mixture at each
level.
The Wendland mixture has maximum support $\ell_J$. We use sparse
factorizations for local supports and dense component-block factorizations
for broad supports.

\subsubsection{Fitting, compression, and evaluation}
\label{sec:overview}

Once the kernel functions satisfy the required constraints, compression
can act on their coefficients. To separate the error introduced by fitting
from the error introduced by compression, let
$\mathcal X$ be a space of sufficiently smooth vector fields, let
$u\in\mathcal X$, and let
$R_X:\mathcal X\to\mathbb R^{dN}$ restrict it to the active nodes. We write
compression as
\begin{equation}
 R_Xu\xrightarrow{\ E\ }c\xrightarrow{\ Q\ }\widetilde c
 \xrightarrow{\ D\ }\widehat u.
 \label{eq:compression-factorization}
\end{equation}
Here, the fitting map $E:\mathbb R^{dN}\to\mathbb R^m$ computes coordinates in a
chosen reconstruction space. The map $Q:\mathbb R^m\to\mathbb R^m$ denotes
lossy compression followed by decompression, and the evaluation map
$D:\mathbb R^m\to\mathcal X$ reconstructs a field. For these kernel methods, $D$ is defined on the continuous domain.
We measure storage using the compressed stream produced within $Q$.
Some methods rescale the fitted coefficients before compression and undo
that scaling during evaluation. We include these operations in $E$ and
$D$, respectively. In \eqref{eq:compression-factorization}, $c$ denotes
the numbers sent to $Q$; the constructions below specify their scaling.

Let $L:\mathcal X\to\mathcal Y$ collect the homogeneous linear constraints to
be preserved, with $\mathcal Y$ the space of constraint values. Here $L$
returns the divergence, the difference between values on identified
periodic boundaries, or both. Every reconstructed field satisfies the
constraints if every function in the range of $D$ satisfies them:
\begin{equation}
 \operatorname{range}(D)\subseteq\ker L,
 \qquad\text{equivalently}\qquad LD=0.
 \label{eq:admissible-decoder}
\end{equation}
Replacing $c$ by $\widetilde c$ adds compression error to the fitting error
already present in $Dc$. A reduced space may therefore have a nonzero
approximation error even when its coefficients are stored exactly.

\paragraph{Direct nodal storage}
\label{sec:direct-nodes}

Direct nodal compression is the reference against which we measure the
storage benefit or cost of a kernel representation. We compress the
entries of $u_X\in\mathbb R^{dN}$ directly, with identity
fitting and evaluation maps and zero approximation error before compression.
This representation defines nodal values only. We also consider a
periodicity-preserving variant, called periodic-node storage: it removes
duplicate endpoint values before compression and copies their reconstructed
values to the paired endpoints afterward. This enforces equality at those
pairs without imposing a divergence constraint.

\subsection{Coefficient representations}

Global interpolation gives an exact nodal fit with one vector coefficient
per node. We first reduce this coefficient count through \nyst{}
approximation. We then use residual interpolation to recover accuracy
when a reduced fit is insufficient. Finally, we fit several kernel spaces
simultaneously with a redundant frame. These choices change both the
fitting problem and the coefficients presented to the compressor.

\subsubsection{Global interpolation and reduced \nyst{} approximation}
\label{sec:global-reduced}

\paragraph{Global Mat\'ern interpolation}
\label{sec:global-matern}

We begin with global interpolation, which replaces the nodal velocities
by an equal number of kernel coefficients. This gives an exact fit and
allows us to assess how the change of representation affects compression.
We seek a kernel expansion $s$ centered at $X$ with $s(x_i)=u(x_i)$.
For a strictly positive-definite matrix kernel this interpolant
is unique and has the smallest native-space norm among functions satisfying the nodal
conditions~\cite{Lowitzsch2005,Wendland2005}. Expanding $s$ in kernel
translates gives
\begin{equation}
 Kc=u_X,
 \qquad
 (Dc)(x)=\sum_{j=1}^{N}\Phi(x,x_j)c_j.
 \label{eq:global-system}
\end{equation}
Here $K=K_\Phi(X,X)\in\mathbb R^{dN\times dN}$ is a square,
dense, symmetric positive-definite (SPD) matrix that couples the velocity components.
The interpolant uses $dN$ coefficients; the computed nodal fit matches the
data to the linear-solve tolerance before compression.
We solve the coupled system by block Cholesky with Schur complements,
followed by iterative refinement. Each block has size $N\times N$.
To set the scale of coefficient errors during
evaluation, we multiply the coefficient vector by a power-iteration estimate
of the largest eigenvalue of $K$ and divide by that estimate when evaluating
the field. The error estimates in \cref{sec:guarantees} use a separate
upper bound on $\|K\|_2$. Kernel parameters and solver stopping criteria
are given in \cref{sec:representation-parameters,sec:implementation}.

\paragraph{Two-level Mat\'ern \nyst{} approximation}
\label{sec:two-level-nystrom}

The global fit requires a dense system with $dN$ unknowns. To reduce this
cost and the number of stored coefficients, we select a subset of sampling
nodes as kernel centers, called landmarks, and fit the nodal velocities in
their span~\cite{WilliamsSeeger2001}. The choice of centers
affects the approximation error; leverage-score and geometric selection
strategies have been studied in
\cite{MuscoMusco2017,OglicGaertner2017}. We use geometric selection for
the first level and residual information to place a second level where
the first fit is least accurate. Both levels remain reduced spaces, so
the construction does not generally interpolate the data.
Let $Z_1\subset X$ contain $M_1$ landmarks. We form their interpolation
matrix $W_1$, its Cholesky factor $L_1$, and the matrix $C_1$ that evaluates
their kernel functions at all sampling nodes:
\begin{equation}
 W_1=K_\Phi(Z_1,Z_1)=L_1L_1^T,
 \qquad
 C_1=K_\Phi(X,Z_1),
 \qquad
 F_1=C_1L_1^{-T}.
 \label{eq:nystrom-features}
\end{equation}
The matrix $W_1\in\mathbb R^{dM_1\times dM_1}$, called the landmark Gramian, is dense and
SPD. Multiplication by $L_1^{-T}$ normalizes the kernel functions in the
native-space inner product. Their sampled values form $F_1$, which has
$dN$ rows and $dM_1$ columns. When $M_1<N$, we fit these columns to more
nodal values than there are coefficients. To obtain an orthonormal basis
for the fit and discard numerically dependent directions, we use a thin
QR factorization followed by a singular value decomposition (SVD)
\cite{GolubVanLoan2013}. Writing the factorizations as $F_1=Q_1R_1$ and
$R_1=P_1\Sigma_1V_1^T$, $Q_1$ has orthonormal columns and $R_1$ is
upper triangular; $P_1$ and $V_1$ contain the left and right singular
vectors of $R_1$. We retain singular values satisfying
\begin{equation}
 \sigma_i>\sigma_1\,\epsilon_{64}\max(dN,dM_1),
 \label{eq:nystrom-rank-threshold}
\end{equation}
where $\epsilon_{64}$ is double-precision machine epsilon. For $r$ retained singular
values, let $P_{1,r}$ and $V_{1,r}$ denote the associated singular
vectors, and $\Sigma_{1,r}$ the actual singular values. The matrix
$U_1=Q_1P_{1,r}$ has $r_1=r$ orthonormal columns at the sampling nodes.
Applying the same change of basis to the kernel functions gives
\begin{equation}
 \Psi_1(x)=K_\Phi(x,Z_1)L_1^{-T}V_{1,r}\Sigma_{1,r}^{-1},
 \qquad R_X\Psi_1=U_1.
 \label{eq:nystrom-continuous-basis-one}
\end{equation}
Because $R_X\Psi_1=U_1$, the coefficients minimizing the nodal
least-squares error are obtained by orthogonal projection. The fitted
coefficients and the remaining residual are
\begin{equation}
 c_1=U_1^Tu_X,
 \qquad r_{\rm first}=u_X-U_1c_1.
 \label{eq:nystrom-level-one}
\end{equation}

The first landmark set covers the domain by farthest-point sampling.
We select a disjoint second set $Z_2$ of $M_2$ landmarks using residuals
from separate training velocity fields, favoring large errors while spreading the
centers spatially. Both sets are fixed before testing; the selection rule
is given in \cref{sec:representation-parameters}.
For $Z_2$, we form $W_2=L_2L_2^T$, $C_2$, and $F_2=C_2L_2^{-T}$
as in \eqref{eq:nystrom-features}. Before fitting the residual,
we remove the components of its columns already represented by $U_1$.
Orthonormalizing the remaining columns gives a second space orthogonal
to the first:
\begin{equation}
 \widetilde F_2=(I-U_1U_1^T)F_2,
 \qquad
 U_2=\operatorname{orth}(\widetilde F_2),
 \qquad
 c_2=U_2^Tr_{\rm first}.
 \label{eq:nystrom-level-two}
\end{equation}
Here $\operatorname{orth}$ uses the same QR--SVD procedure and truncation
rule as at the first level, leaving $r_2$ orthonormal columns.
Let $T_2$ be the resulting change-of-basis
matrix, so that $U_2=\widetilde F_2T_2$. To evaluate the second correction
as a continuous field, we apply the corresponding projection and change
of basis to the kernel functions:
\begin{equation}
 \begin{aligned}
 \Psi_2(x)&=\left[G_2(x)-\Psi_1(x)U_1^TG_2(X)\right]T_2,\\
 G_2(x)&=K_\Phi(x,Z_2)L_2^{-T},
 & R_X\Psi_2&=U_2.
 \end{aligned}
 \label{eq:nystrom-continuous-basis-two}
\end{equation}
Both $\Psi_1$ and $\Psi_2$ are linear combinations of divergence-free
kernel columns, with periodic columns used for Taylor--Green. The coefficient-fitting and field-evaluation maps are
therefore
\begin{equation}
 E(u_X)=\begin{bmatrix}U_1^Tu_X\\U_2^T(I-U_1U_1^T)u_X\end{bmatrix},
 \qquad
 D(c_1,c_2)(x)=\Psi_1(x)c_1+\Psi_2(x)c_2.
 \label{eq:two-level-nystrom}
\end{equation}
In exact arithmetic, $U_1^TU_2=0$ and the nodal evaluation matrix has operator
norm one. The component outside $\operatorname{range}[U_1\ U_2]$
remains as approximation error. Its Euclidean norm is
\begin{equation}
 \left\|(I-U_1U_1^T-U_2U_2^T)u_X\right\|_2.
 \label{eq:nystrom-projection-error}
\end{equation}
This error is present before compression and cannot be reduced by storing
the same coefficients more accurately. Kernel parameters and landmark
counts are given in \cref{sec:representation-parameters}; the retained
rank at each level is determined by \eqref{eq:nystrom-rank-threshold}.

\subsubsection{Multilevel residual interpolation}
\label{sec:residual-interpolation}
The reduced spaces above do not strictly interpolate. However, this is easily remedied by interpolating the residual
left by a reduced or coarse fit. We consider two constructions: a
Mat\'ern \nyst{} fit followed by a Wendland correction, and a sequence of
Wendland fits on nested center sets. Both end with a correction at all
active nodes.

\paragraph{\nyst{} + Wendland Hybrid}
\label{sec:hybrid}

The hybrid uses a compact Mat\'ern fit to reduce the residual that must be
interpolated. We then represent that residual with property-preserving
Wendland functions, so both terms retain the required constraints after
compression. Using the first \nyst{}
basis $U_M$ with $r_M$ retained columns, we compute
\begin{equation}
 c_M=U_M^Tu_X,
 \qquad r_{\rm base}=u_X-U_Mc_M.
 \label{eq:hybrid-base}
\end{equation}
Here $r_{\rm base}$ is the residual after projection onto the Mat\'ern landmark space.
Let $K_W=K_{\Phi_W}(X,X)$ correspond to the dataset-specific Wendland kernel, with
support radius $\rho=\beta h_X$ for cavity and Taylor--Green. Here
$h_X$ is the median nearest-neighbor node spacing and $\beta>0$.
Species transport instead uses the fixed physical-scale mixture.
The residual is interpolated by solving the system
\begin{equation}
 K_Wa=r_{\rm base},
 \qquad
 u_X^{\rm fit}=U_Mc_M+K_Wa.
 \label{eq:hybrid-residual}
\end{equation}
The base fit uses the orthogonal projection in \eqref{eq:nystrom-level-one} while the
correction is obtained from a square interpolation problem. Because the Wendland kernel is
positive definite, $K_W$ is SPD and sparse~\cite{Wendland2005,Wendland2009}.
We therefore rely on sparse direct factorization when the supports are local. For species transport, this matrix is nearly dense and thus uses coupled block Cholesky.
Because $K_Wa=r_{\rm base}$, the uncompressed fit satisfies
$u_X^{\rm fit}=U_Mc_M+r_{\rm base}=u_X$ up to linear-solve tolerance, thereby removing the \nyst{} projection error.
Of course, subsequent lossy compression perturbs both coefficient blocks.
To control how coefficient errors affect nodal values, we scale the
Wendland coefficients using an upper bound on the evaluation-matrix norm.
For a matrix $B$, $\|B\|_1$ and $\|B\|_\infty$ are the largest absolute
column and row sums, respectively. We set
\begin{equation}
 s_W=\sqrt{\|K_W\|_1\|K_W\|_\infty},\qquad
 a^{\rm sc}=s_Wa.
 \label{eq:hybrid-scaling}
\end{equation}
We evaluate the scaled coefficients with $K_W/s_W$, whose Euclidean
operator norm is at most one. A perturbation of the scaled coefficient
block therefore produces no larger Euclidean perturbation in the nodal
values. This supplies the constant needed to allocate compression
tolerances in \cref{sec:guarantees}. The value of $\beta$
and the landmark counts are given in \cref{sec:representation-parameters}.

\paragraph{Multilevel Wendland interpolation}
\label{sec:multilevel-wendland}
One can also build the entire fit from Wendland kernels. On nested center
sets, each interpolation system fits the residual from the preceding
levels. Local supports allow sparse systems, offering an alternative to
the dense reduced fits above. Multilevel approximation with compactly supported kernels
has been studied in
\cite{Opfer2006Multiscale,Wendland2010,GeorgoulisEtAl2013,
FarrellWendland2013,GriebelRiegerZwicknagl2015,LeGiaSloanWendland2017},
including divergence-free interpolation~\cite{FarrellGillowWendland2017}.
Multilevel expansions have also been used for compression by coefficient
thresholding on the sphere~\cite{LeGiaWendland2014}. Here we retain the
coefficient arrays and apply the same lossy compressors used for the other
representations.

We fit three corrections in sequence, each to the residual left by the
preceding levels. The final center set contains every node, ensuring
interpolation before compression. Let
$Z_1\subset Z_2\subset Z_3=X$ be nested center sets selected by farthest-point
sampling. For cavity and Taylor--Green, coarse centers have broad supports; finer centers use shorter
supports to correct more localized residual structure. At level $\ell$, a continuous correction
$w_\ell$, formed from Wendland kernels centered at $Z_\ell$, interpolates the current residual on
$Z_\ell$; the correction is then evaluated on all of $X$ before the next
level. For cavity and Taylor--Green, let $h_\ell$ be the median
nearest-neighbor center spacing and set $\rho_\ell=\beta h_\ell$.
At each level, $A_\ell$ evaluates the correction at its own centers,
where interpolation conditions are imposed, and $D_\ell$ evaluates it
at all sampling nodes:
\begin{equation}
 A_\ell=K_{\Phi_W^{(\ell)}}(Z_\ell,Z_\ell),
 \qquad
 D_\ell=K_{\Phi_W^{(\ell)}}(X,Z_\ell).
 \label{eq:wendland-level-matrices}
\end{equation}
Starting from $r_0=u_X$, we solve for $a_\ell$ and update the residual at all nodes:
\begin{equation}
 A_\ell a_\ell=R_\ell r_{\ell-1},
 \qquad
 r_\ell=r_{\ell-1}-D_\ell a_\ell,
 \qquad \ell=1,2,3,
 \label{eq:wendland-multilevel}
\end{equation}
where $R_\ell$ restricts the residual to $Z_\ell$. As in the hybrid approach, each
stored block uses
$s_\ell=\sqrt{\|D_\ell\|_1\|D_\ell\|_\infty}$,
$a_\ell^{\rm sc}=s_\ell a_\ell$, and is evaluated as
$D_\ell(a_\ell^{\rm sc}/s_\ell)$. The uncompressed nodal fit is
\begin{equation}
 u_X^{\rm fit}=\sum_{\ell=1}^3
 D_\ell\frac{a_\ell^{\rm sc}}{s_\ell}.
 \label{eq:wendland-decoder}
\end{equation}
Each $A_\ell$ is a square SPD collocation matrix. For cavity and Taylor--Green,
the decreasing radii keep the matrices sparse as the center set is refined,
and the systems use sparse direct factorization. For species transport we
use the same physical mixture at all three levels and solve the systems
with dense block Cholesky. In both cases,
their factorizations are reused across velocity fields on the same geometry. Because $Z_3=X$, the last level interpolates the remaining nodal
residual, giving an exact uncompressed nodal reconstruction to solver
tolerance. Center counts, $\beta$, and solver tolerances
are specified in \cref{sec:representation-parameters,sec:implementation}.

\subsubsection{Redundant kernel frames}
\label{sec:kernel-frame}

Sequential residual fitting fixes each coefficient block before computing
the next. To allow all blocks to contribute to the interpolation conditions
simultaneously, we use the multiscale kernel frames introduced in~\cite{FriedenEtAl2026} with
the property-preserving kernels of \cref{sec:property-kernels}. The
combined representation has more coefficients than nodal conditions and
therefore admits many interpolating coefficient vectors. Each level contributes a block of coefficients, with one vector coefficient
per center. The frame combines kernel functions from three nested center sets,
giving more coefficients than nodal values.
Following~\cite{FriedenEtAl2026}, let $Z_0=X\supset Z_1\supset Z_2$. If $h_X$ is the
finest-node spacing, the cavity and Taylor--Green support radii are
\begin{equation}
 \rho_\ell=\beta\,2^\ell h_X,
 \qquad \ell=0,1,2.
 \label{eq:frame-radii}
\end{equation}
Let $D_\ell=K_{\Phi_W^{(\ell)}}(X,Z_\ell)$. The matrix of their sampled
values is
\begin{equation}
 A=\begin{bmatrix}D_0&D_1&D_2\end{bmatrix}
 \in\mathbb R^{dN\times m},
 \qquad
 m=d\sum_{\ell=0}^2|Z_\ell|>dN.
 \label{eq:frame-matrix}
\end{equation}
The reconstructed field sums the three kernel expansions. Exact nodal reconstruction requires solving
$dN$ equations for $m>dN$ unknown coefficients
\cite{Opfer2006TightFrame,FriedenEtAl2026}.
The finest block contains every active node, so its positive-definite
collocation matrix makes $A$ full row rank.

We choose the unique minimum-norm coefficient vector for this unscaled matrix:
\begin{equation}
 c^\star=\arg\min_{Ac=u_X}\|c\|_2,
 \qquad
 c^\star=A^T(AA^T)^{-1}u_X.
 \label{eq:frame-coefficients}
\end{equation}
For local supports, we compute these coefficients by enforcing both the
interpolation constraint and the stationarity condition for minimizing
$\tfrac12\|c\|_2^2$. With Lagrange multiplier $\lambda$, these conditions give
\begin{equation}
 \begin{bmatrix}I&A^T\\A&0\end{bmatrix}
 \begin{bmatrix}c^\star\\\lambda\end{bmatrix}
 =\begin{bmatrix}0\\u_X\end{bmatrix}.
 \label{eq:frame-kkt}
\end{equation}
The matrix is square,
sparse, symmetric, and indefinite. The identity block selects the Euclidean
minimum-norm solution from $Ac=u_X$. Rescaling the columns before fitting
would change this minimizer.
For species transport the same kernel is used at every level, so
$A=KS$, where $S$ concatenates the componentwise center-selection matrices.
Each column of $S$ selects a column of $K$, corresponding to one center
and velocity component. Thus $H=SS^T$ is diagonal, and each diagonal entry
counts how many times that kernel column appears across the levels. All entries are
positive because the finest level contains every node. Substituting
$A=KS$ into \eqref{eq:frame-coefficients} gives
\begin{equation}
 c^\star=S^TH^{-1}K^{-1}u_X.
\label{eq:frame-shared-kernel}
\end{equation}
We evaluate this formula with one coupled SPD block solve and diagonal
weighting. The coarse blocks repeat columns of the fine block, and the
minimum-norm fit divides each interpolation coefficient among its copies.
\begin{table}[!htpb]
\caption{Coefficient-fitting systems and uncompressed coefficient counts per
velocity field. Here $n=dN$, $r_j$ is the retained rank of landmark level $j$,
$r_M$ is the retained rank of the hybrid's Mat\'ern fit,
and SPD means symmetric positive definite.
Factorizations are reused on fixed geometry; coefficient counts do not
determine compressed byte counts.}
\label{tab:method-cost}
\centering
\small
\begin{tabular}{@{}>{\raggedright\arraybackslash}p{0.25\linewidth}>{\raggedright\arraybackslash}p{0.45\linewidth}l@{}}
\toprule
Method & Fitting system & Coefficients \\
\midrule
Direct nodes & None & $n$ \\
Global Mat\'ern & Dense SPD interpolation & $n$ \\
Two-level \nyst{} & Landmark SPD systems; QR--SVD projection & $r_1+r_2$ \\
\nyst{} + Wendland & Projection; SPD correction & $r_M+n$ \\
Multilevel Wendland & SPD residual interpolation; sparse for local supports & $\sum_\ell d|Z_\ell|$ \\
Kernel frame & Underdetermined minimum-norm fit & $\sum_\ell d|Z_\ell|$ \\
\bottomrule
\end{tabular}
\end{table}
All physical scales enter through the mixture defining $K$.
Solver details and residual tolerances are given in
\cref{sec:implementation}. After this fit, each block is scaled by
$s_\ell=\sqrt{\|D_\ell\|_1\|D_\ell\|_\infty}$; the stored coordinates are
$c_\ell^{\rm sc}=s_\ell c_\ell^\star$, and evaluation uses $D_\ell/s_\ell$.
The three coefficient blocks are compressed separately for cavity and
Taylor--Green. For species transport they are concatenated in level and component order
and compressed as one array; \cref{sec:guarantees-ablations} compares these storage choices.
The center counts, radii, and physical lengths are listed in
\cref{sec:representation-parameters}.

\Cref{tab:method-cost} summarizes the fitting systems and coefficient counts.
We rebuild the fitting and evaluation matrices from the geometry, center
or landmark indices, kernel parameters, and coefficient-block sizes.
These quantities are shared across velocity fields on a fixed mesh.
We therefore do not store a dense evaluation matrix with each compressed
field. Compression ratios count the bytes stored for each velocity field;
\cref{sec:experiments} specifies the shared setup.

%% file: guarantees.tex
\section{Preservation and nodal error estimates}
\label{sec:guarantees}

Compression perturbs the coefficients computed in \cref{sec:methods}.
We now show how those perturbations affect the reconstructed field.
The linear constraints remain satisfied because the basis functions are
unchanged. The nodal values do change, and their error depends on both
the coefficient error and the evaluation matrix. We bound this change
for all five representations, then use orthogonality to obtain a sharper
estimate for two-level \nyst{} approximation.

\subsection{Preservation under coefficient perturbations}

\begin{lemma}[preservation by the evaluation map]
\label{thm:preservation}
Let $D:\mathbb R^m\to\mathcal X$ and $L:\mathcal X\to\mathcal Y$ be linear,
with $LD=0$. For any coefficient map $Q:\mathbb R^m\to\mathbb R^m$,
\begin{equation}
 LDQ(c)=0\qquad(c\in\mathbb R^m).
 \label{eq:nullspace}
\end{equation}
\end{lemma}
\begin{proof}
Every reconstructed field $DQ(c)$ belongs to
$\operatorname{range}(D)\subseteq\ker L$.
\end{proof}

For the kernels in \cref{sec:methods}, take $L$ to be divergence or the
difference between field values on paired periodic boundaries. The lemma
applies regardless of the compressor tolerance. To control velocity
accuracy as well, we need a bound on the perturbation of the coefficients.

\subsection{Error induced in the nodal values}

In this section, let $A=R_XD=[A_1,\ldots,A_J]$ evaluate the stored coefficient blocks at
the nodes. Each block includes the inverse of any coefficient scaling
applied before compression. Write $c$ for the uncompressed coefficients
and $\widetilde c=c+\delta c$ for their decompressed values. Their
nodal reconstructions differ by $A\delta c$. We use $\|\cdot\|_2$ for
the Euclidean vector norm and its induced matrix norm, so that
$\|A\delta c\|_2\leq\|A\|_2\|\delta c\|_2$.

\begin{theorem}[nodal coefficient-to-field bound]
\label{thm:codec-bound}
Suppose coefficient block $j$ has $m_j$ entries and
$b_j\geq\|A_j\|_2$. With
$E_{\rm rep}=\|u_X-Ac\|_2$, the total nodal error satisfies
\begin{align}
 u_X-A\widetilde c&=(u_X-Ac)-A\delta c,
 \label{eq:nodal-error-identity}\\
 \|u_X-A\widetilde c\|_2
 &\leq E_{\rm rep}+\sum_{j=1}^J b_j\|\delta c_j\|_2.
 \label{eq:nodal-block-bound}
\end{align}
If each block satisfies a maximum-absolute or root-mean-square (RMS)
coefficient-error bound,
\begin{equation}
 \|\delta c_j\|_\infty\leq\varepsilon_j
 \quad\hbox{or}\quad
 m_j^{-1/2}\|\delta c_j\|_2\leq\varepsilon_j,
 \label{eq:codec-native-contracts}
\end{equation}
then
\begin{equation}
 \|u_X-A\widetilde c\|_2
 \leq E_{\rm rep}+\sum_{j=1}^J b_j\sqrt{m_j}\,\varepsilon_j.
 \label{eq:codec-coefficient-bound}
\end{equation}
In particular, choosing coefficient tolerances satisfying
\begin{equation}
 \sum_{j=1}^J b_j\sqrt{m_j}\,\varepsilon_j\leq\tau
 \label{eq:codec-budget-allocation}
\end{equation}
limits the additional nodal error to $\tau$.
\end{theorem}
\begin{proof}
Add and subtract $Ac$, apply the triangle inequality to
$\sum_j A_j\delta c_j$, and use
$\|A_j\delta c_j\|_2\leq b_j\|\delta c_j\|_2$.
Either assumption in \cref{eq:codec-native-contracts} implies
$\|\delta c_j\|_2\leq\sqrt{m_j}\varepsilon_j$.
\end{proof}

In \eqref{eq:codec-coefficient-bound}, $E_{\rm rep}$ is the error already
present in the fit, and $b_j$ bounds the amplification of coefficient
errors during evaluation. Once these quantities are known, the requested
coefficient tolerances bound the total nodal error. Here $\tau$ is an
absolute Euclidean allowance for the compression contribution. Dividing
the bound by $\|u_X\|_2$ gives a relative error bound; any fitting error
must still be included when prescribing a total relative error.
The two assumptions in \eqref{eq:codec-native-contracts} accommodate
maximum-absolute and RMS error controls without requiring a particular
compression algorithm. The estimate applies when the compressor meets
the stated coefficient tolerance. The modes used in our experiments are
specified in \cref{sec:experiments}.

\subsection{Bounds for the coefficient representations}
\label{sec:representation-bounds}

The evaluation matrices in \cref{sec:methods} determine the constants in
\cref{thm:codec-bound}. Orthonormal columns preserve the Euclidean size
of a coefficient perturbation. For the interpolation and frame matrices,
upper bounds on their norms give sufficient coefficient tolerances.

\paragraph{Orthonormal landmark spaces}
For two-level \nyst{} in \cref{sec:two-level-nystrom}, set $U=[U_1\ U_2]$,
$c=U^Tu_X$, and $e_{\rm rep}=(I-UU^T)u_X$. The projection residual is
orthogonal to the columns of $U$, whereas the compression perturbation
$U\delta c$ lies in their span. In exact arithmetic, $U^TU=I$, so
the Pythagorean identity gives
\begin{equation}
 \|u_X-U(c+\delta c)\|_2^2
 =\|e_{\rm rep}\|_2^2+\|\delta c\|_2^2
 \leq E_{\rm rep}^2+\sum_jm_j\varepsilon_j^2.
 \label{eq:nystrom-nodal-bound}
\end{equation}
Storing the two coefficient blocks as one vector with tolerance
$\varepsilon=\tau/\sqrt{m_1+m_2}$ gives
\begin{equation}
 \|u_X-U(c+\delta c)\|_2
 \leq\sqrt{E_{\rm rep}^2+\tau^2}.
 \label{eq:nystrom-budget-bound}
\end{equation}
Increasing the landmark count can reduce $E_{\rm rep}$; tightening the
compressor tolerance reduces the bound on the compression error. In floating-point arithmetic,
these identities hold up to fitting, orthogonality, and evaluation errors.

\paragraph{Scaled interpolation blocks}
The interpolation and frame matrices need not have orthonormal columns.
We control their norms by scaling. For an unscaled block $B_j$, the normalization
$q_j=\sqrt{\|B_j\|_1\|B_j\|_\infty}$ gives
$A_j=B_j/q_j$ with $\|A_j\|_2\leq1$. This is the scaling used for the
Wendland and frame blocks in \cref{sec:methods}. For $J$ such blocks,
\begin{equation}
 \|[A_1\ \cdots\ A_J]\|_2\leq\sqrt J.
\label{eq:concatenated-block-bound}
\end{equation}
For block vectors $v=(v_1^T,\ldots,v_J^T)^T$, the triangle inequality and
Cauchy--Schwarz give
\[
 \left\|\sum_j A_jv_j\right\|_2\leq\sum_j\|v_j\|_2
 \leq\sqrt{J}\left(\sum_j\|v_j\|_2^2\right)^{1/2}
 =\sqrt{J}\|v\|_2.
\]
Thus a common coefficient tolerance
$\varepsilon=\tau/\sqrt{Jm}$, where $m=\sum_jm_j$, bounds the added error
by $\tau$, whether the blocks are compressed together or separately.
One can instead allocate fractions
$w_j\geq0$, $\sum_jw_j\leq1$, and set
$\varepsilon_j=w_j\tau/\sqrt{m_j}$ as in
\cref{eq:codec-budget-allocation}. The experimental allocations are
specified in \cref{sec:coefficient-tolerances}.

\paragraph{Global interpolation}
We scale the global coefficients by a power-iteration estimate
$\widehat\lambda$ and evaluate them with $A=K/\widehat\lambda$.
Power iteration alone does not prove that $\|A\|_2\leq1$.
An upper bound based on absolute row and column sums supplies the required
constant for the same stored coefficients:
\begin{equation}
 b_{\rm global}=
 \frac{\sqrt{\|K\|_1\|K\|_\infty}}{\widehat\lambda}
 \ \geq\ \|K/\widehat\lambda\|_2.
 \label{eq:global-valid-bound}
\end{equation}
Choosing $\varepsilon=\tau/(b_{\rm global}\sqrt{dN})$ therefore bounds
the additional error by $\tau$. The experiments use
$\varepsilon=\tau/\sqrt{dN}$ after power-iteration scaling, giving an
allowance of $b_{\rm global}\tau$. We evaluate this constant for cavity and
Taylor--Green in \cref{sec:bound-details}; no corresponding constant is reported for
the species physical-scale mixture.

All bounds concern the stored nodes. In numerical interpolation,
$E_{\rm rep}$ includes the solve residual. Floating-point evaluation
introduces an additional roundoff error.

%% file: results.tex
\section{Results}
\label{sec:results}

We now ask how much storage each representation saves at a prescribed
velocity accuracy, and which field properties survive compression.
The three flows introduced in \cref{sec:introduction} test reduced fitting
and interpolation on smooth, localized, and multiscale velocity fields.
All methods use the same five velocity fields for each flow.
We first verify the nodal
error bounds, then compare compression at matched velocity accuracy and
the properties of the reconstructed fields. Dataset details and numerical
settings appear in \cref{sec:experiments}.

\subsection{Quantities reported}
\label{sec:results-quantities}

\paragraph{Velocity error and storage}
Let $u_X^{(s)},\widehat u_X^{(s)}\in\mathbb R^{dN}$ denote the original and
reconstructed nodal velocities for field $s=1,\ldots,S$, with $S=5$. If
$B_{\rm raw}$ counts bytes in the original velocity array and
$B_{\rm comp}^{(s)}$ counts bytes in its compressed representation, the
compression ratio and relative discrete field error are
\begin{equation}
 R^{(s)}=\frac{B_{\rm raw}}{B_{\rm comp}^{(s)}},\qquad
 e_2^{(s)}=\frac{\|\widehat u_X^{(s)}-u_X^{(s)}\|_2}
 {\|u_X^{(s)}\|_2}.
 \label{eq:results-rate-error}
\end{equation}
The rate--distortion figures plot the mean compression ratio against the
largest relative error over those fields:
\begin{equation}
 \overline R=\frac1S\sum_{s=1}^S R^{(s)},\qquad
 E_2=\max_{1\leq s\leq S}e_2^{(s)}.
 \label{eq:aggregate-rate-error}
\end{equation}
The maximum pointwise-error curves in
\cref{sec:codec-appendix} instead use
\begin{equation}
 E_\infty=\max_s\max_i
 \|\widehat u^{(s)}(x_i)-u^{(s)}(x_i)\|_2.
\label{eq:maximum-pointwise-error}
\end{equation}
Here $E_\infty$ is absolute, in velocity units, whereas $E_2$ is relative.
For each method and landmark count, the main curves select the compressor
and setting with the largest $\overline R$ satisfying
$E_2\leq\varepsilon_{\rm field}$ at each velocity-error limit
$\varepsilon_{\rm field}$. Each setting applies to all five fields;
the selected compressor can change along a curve.
\Cref{sec:codec-appendix} gives comparisons with each compressor fixed.

For Taylor--Green, direct compression stores one velocity per distinct
periodic node. The periodicity experiment also compresses duplicate
boundary values independently to test the loss of endpoint equality.

\paragraph{Incompressibility and periodicity}
We measure the divergence introduced by direct nodal compression relative
to the uncompressed data, using a common local polyharmonic-spline (PHS)
differentiation operator $\mathcal D_h$:
\begin{equation}
 \delta_{\rm div}^{(s)}=\frac{L_X}{U_{\rm rms}^{(s)}}
 \left(\frac{1}{|\mathcal I|}\sum_{i\in\mathcal I}
 |\mathcal D_h(\widehat u^{(s)}-u^{(s)})(x_i)|^2\right)^{1/2}.
 \label{eq:divergence-measure}
\end{equation}
Here $\mathcal I$ contains the interior evaluation nodes, $L_X$ is the largest
coordinate range, and $U_{\rm rms}^{(s)}=\|u_X^{(s)}\|_2/\sqrt N$.
The PHS operator is specified in \cref{sec:experiments}. Applying it to the
velocity difference removes the discrete divergence of the original data.
Kernel reconstructions instead satisfy $\nabla\cdot\widehat u=0$ analytically.
Their plotted zeros express this guarantee; the direct-point values
estimate divergence change using nodal differentiation.
For Taylor--Green, let $(x_p^+,x_p^-)$, $p=1,\ldots,P$, denote pairs
of boundary nodes representing the same periodic location. Periodicity
requires equal velocities at each pair. We measure their RMS and maximum
mismatches by
\begin{equation}
 \begin{aligned}
 \delta_{\pi,2}^{(s)}
 &=\frac{1}{U_{\rm rms}^{(s)}}
 \left(\frac1P\sum_{p=1}^P
 \|\widehat u^{(s)}(x_p^+)-\widehat u^{(s)}(x_p^-)\|_2^2\right)^{1/2},\\
 \delta_{\pi,\infty}^{(s)}
 &=\frac{\max_p\|\widehat u^{(s)}(x_p^+)-\widehat u^{(s)}(x_p^-)\|_2}
 {\max_i\|u^{(s)}(x_i)\|_2}.
 \end{aligned}
 \label{eq:periodicity-measures}
\end{equation}

\paragraph{Velocity fluctuations across scales}
For species transport, we measure changes in velocity fluctuations across
scales using the second-order structure function. For node pairs $I_k$
whose separation lies in the bin centered at $r_k$, it is
\begin{equation}
 S_2^{(s)}(r_k)=\frac{1}{|I_k|}\sum_{(i,j)\in I_k}
 \|u^{(s)}(x_i)-u^{(s)}(x_j)\|_2^2.
 \label{eq:structure-function}
\end{equation}
With $\boldsymbol S_2=(S_2(r_k))_k$ and its reconstructed counterpart
$\widehat{\boldsymbol S}_2$, the relative change is
\begin{equation}
 \frac{\|\widehat{\boldsymbol S}_2-\boldsymbol S_2\|_2}
      {\|\boldsymbol S_2\|_2}.
\label{eq:structure-function-error}
\end{equation}
Suppressing the field index, we use bins covering
$[0.0041304,0.0179020]$ in source coordinate units and fit
\begin{equation}
 S_2(r)=Cr^p
 \label{eq:structure-exponent}
\end{equation}
over these bins, with amplitude $C$ and exponent $p$. We compare $p$
with the uncompressed value and the Kolmogorov prediction
$2/3$~\cite{Kolmogorov1991}. The source simulation uses a turbulence
closure, so this comparison tests sampled velocity scaling without
establishing a resolved inertial cascade.

The property plots average these measures over the five test velocity fields,
using only settings that meet the velocity-error limit on every field.

\subsection{Verification of the nodal coefficient-error bound}
\label{sec:bound-verification}

We choose coefficient tolerances from \cref{thm:codec-bound} for a nodal
error allowance $\tau$. For the normalized landmark, Wendland, hybrid,
and frame representations,
\begin{equation}
 \theta=\frac{\|A\delta c\|_2}{\tau}
 \label{eq:bound-effectivity}
\end{equation}
measures the fraction of that allowance used: values near one indicate
a sharp estimate, while small values indicate a conservative one.
For frame fits, we separate the uncompressed residual
$\rho=\|Ac-u_X\|_2$ from the total error $e=\|A\widetilde c-u_X\|_2$ using
\begin{equation}
 \max\{0,e-\rho\}\leq\|A\delta c\|_2\leq e+\rho.
\label{eq:compression-error-bracket}
\end{equation}
For species transport, the compression-only error is measured
directly by subtracting the uncompressed fitted field from the decompressed field.

\begin{table}[!htbp]
\centering
\caption{Observed nodal compression error divided by its predicted upper
bound $\tau$, as defined in \eqref{eq:bound-effectivity}. The predicted bound
corresponds to 1; values below 1 satisfy it. Statistics pool the species-transport
CPU compressors, landmark counts, and tolerances, excluding cases that
exceed the requested coefficient error.}
\label{tab:decoder-bound-verification}
\begin{tabular}{lrr}
\toprule
Method & \shortstack{Median\\observed/bound} & \shortstack{Maximum\\observed/bound}\\
\midrule
Two-level \nyst{} & 0.159 & 0.635\\
\nyst{} + Wendland Hybrid & 0.013 & 0.120\\
Multilevel Wendland & 0.0027 & 0.0072\\
Frame Kernel Method & 0.0028 & 0.0073\\
\bottomrule
\end{tabular}
\end{table}

Across the tested flows and landmark counts, the nodal compression error
remained below the applicable estimate whenever the compressor met its
coefficient bound. The estimate was closest for two-level \nyst{}
(\cref{tab:decoder-bound-verification}), whose orthonormal evaluation
preserves the norm of the coefficient perturbation. The other bounds use
upper estimates of the evaluation matrix norm and therefore need not be
attained by the actual coefficient perturbations. Observed errors can
also be smaller because compressors return coefficient errors below
their prescribed tolerances.

For two-level \nyst{}, the total error also contains the projection error
$E_{\rm rep}$. Orthogonality gives the estimate
$\sqrt{E_{\rm rep}^2+\tau^2}$ in \cref{eq:nystrom-budget-bound}.
The hybrid's full-node Wendland correction removes this approximation error
up to the interpolation solve accuracy.

Independent coefficient perturbations also verify the
\nyst{} total-error identity \eqref{eq:nystrom-nodal-bound} to
floating-point accuracy. \Cref{sec:bound-details} gives the norm calculations,
flow-specific values, and excluded cases.

\input{results_rate_distortion}
\input{results_property_preservation}

%% file: results_rate_distortion.tex
\subsection{Field accuracy and compression rate}
\label{sec:results-rate-distortion}

We compare storage at matched velocity accuracy across the three flows.
Reduced \nyst{} fits use fewer coefficients but may leave a fitting error;
interpolating methods remove this limitation. The plots span several
velocity-error limits. We use $10^{-3}$ as a common accuracy for the
comparisons highlighted here and in \cref{sec:property-results}; it does
not restrict the kernel guarantees, which hold at every compression tolerance.

\begin{figure}[!htbp]
 \centering \includegraphics[page=1,width=0.8\linewidth]{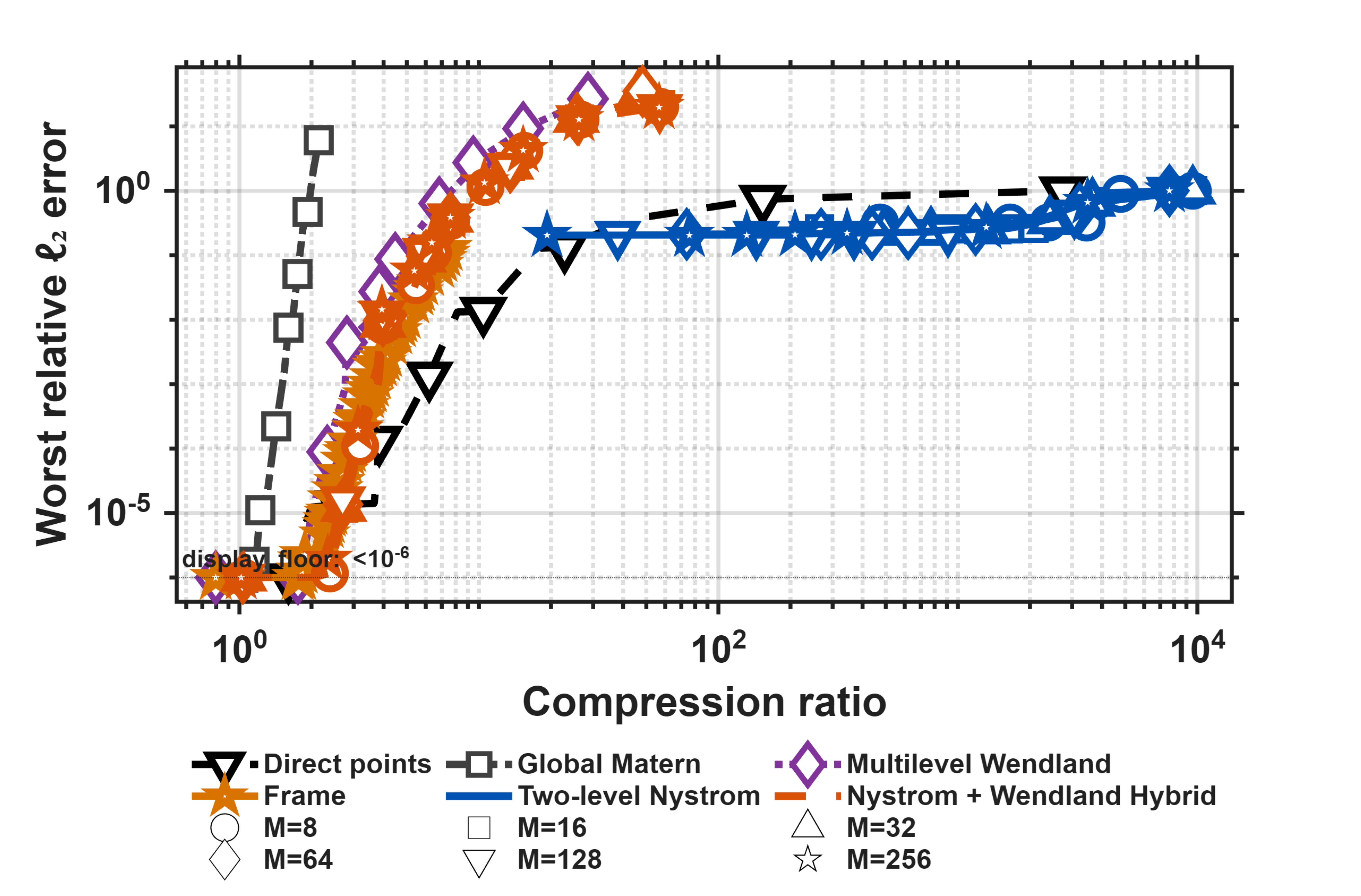}
 \caption{Lid-driven cavity. All tested landmark counts for two-level \nyst{}
 and the \nyst{} + Wendland Hybrid are shown with direct points, global Mat\'ern,
 multilevel Wendland, and the Frame Kernel Method. Each curve is the
 Pareto envelope across four compressors; each point uses the mean compression
 ratio and largest error over five velocity fields, defined in
 \eqref{eq:aggregate-rate-error}. Errors below
 $10^{-6}$ are placed on the labeled display floor.}
 \label{fig:rate-distortion-lid}
\end{figure}

\paragraph{Comparison across compressors}
The curves combine ZFP, PFPL, SZ3, and MGARD using the selection rule in
\cref{sec:results-quantities}, after checking each compressor's prescribed
error criterion on all five velocity fields. We retain a setting only if no other
setting gives at least as much compression with no larger error and
improves one of the two quantities. These retained settings form the
Pareto envelope; successive points may use different compressors.
Most reported $10^{-3}$ comparisons use PFPL; the two-level Taylor--Green
\nyst{} point and several species points use ZFP. The best compressor varies
with the coefficient distribution. Separate compressor curves and
pointwise-error comparisons appear in \cref{sec:codec-appendix}.

\paragraph{Landmark counts}

Increasing the landmark count gives two-level \nyst{} a larger fitting
space at the cost of more coefficients. The hybrid already interpolates
with its all-node correction, so changing its base count instead changes
how the data are divided between the two coefficient blocks. We evaluate both methods at
the landmark counts listed in \cref{sec:representation-parameters}. For two-level
\nyst{}, each level uses up to $M$ landmarks, with disjoint sets for the two levels.
For the hybrid, only the Mat\'ern
base landmark count changes; the Wendland correction continues to use every active node.
The curve parameter $M$ denotes the requested landmark count, not the retained
matrix rank $r_j\leq dM_j$. The figures compare all tested landmark counts with the
other representations at the same velocity accuracy.
The hybrid interpolates the nodal values to solver tolerance before
compression. Its plotted error comes from lossy coefficient compression
and finite-precision fitting and evaluation.

\paragraph{Lid-driven cavity}
\Cref{fig:rate-distortion-lid} shows that reducing the coefficient count
through the tested landmark spaces leaves a substantial cavity fitting error.
Even the largest tested space has a relative error above $0.2$, far
above the $10^{-3}$ error limit. The plateau is already present in the
uncompressed fit: tightening the coefficient tolerance cannot recover
velocity components outside the retained space. The full Wendland correction
removes this approximation error. Once it is added, changing the base landmark count has
little effect on compressed size at the $10^{-3}$ error limit.

At this $10^{-3}$ error limit, the frame gives the largest compression ratio
among the tested divergence-free methods, followed by the hybrid. Both
remain below direct nodal compression. Despite its larger coefficient
count, the frame requires the least compressed storage among the kernel
representations at this accuracy.

\paragraph{Taylor--Green flow}
In \cref{fig:rate-distortion-taylor}, two-level periodic \nyst{} resolves all five Taylor--Green velocity fields below
relative error $10^{-3}$. At this error limit, $M=32$ gives the largest
tested two-level compression ratio, $2.92\times10^3$. Increasing the landmark count lowers the
fitting error but also enlarges the coefficient vector. The hybrid
is less sensitive to landmark count and requires more storage.
At this error limit, the two-level reduced space is already sufficient;
the hybrid's exact uncompressed interpolation does not compensate for
the additional stored correction.

\begin{figure}[!htbp]
 \centering \includegraphics[page=2,width=0.80\linewidth]{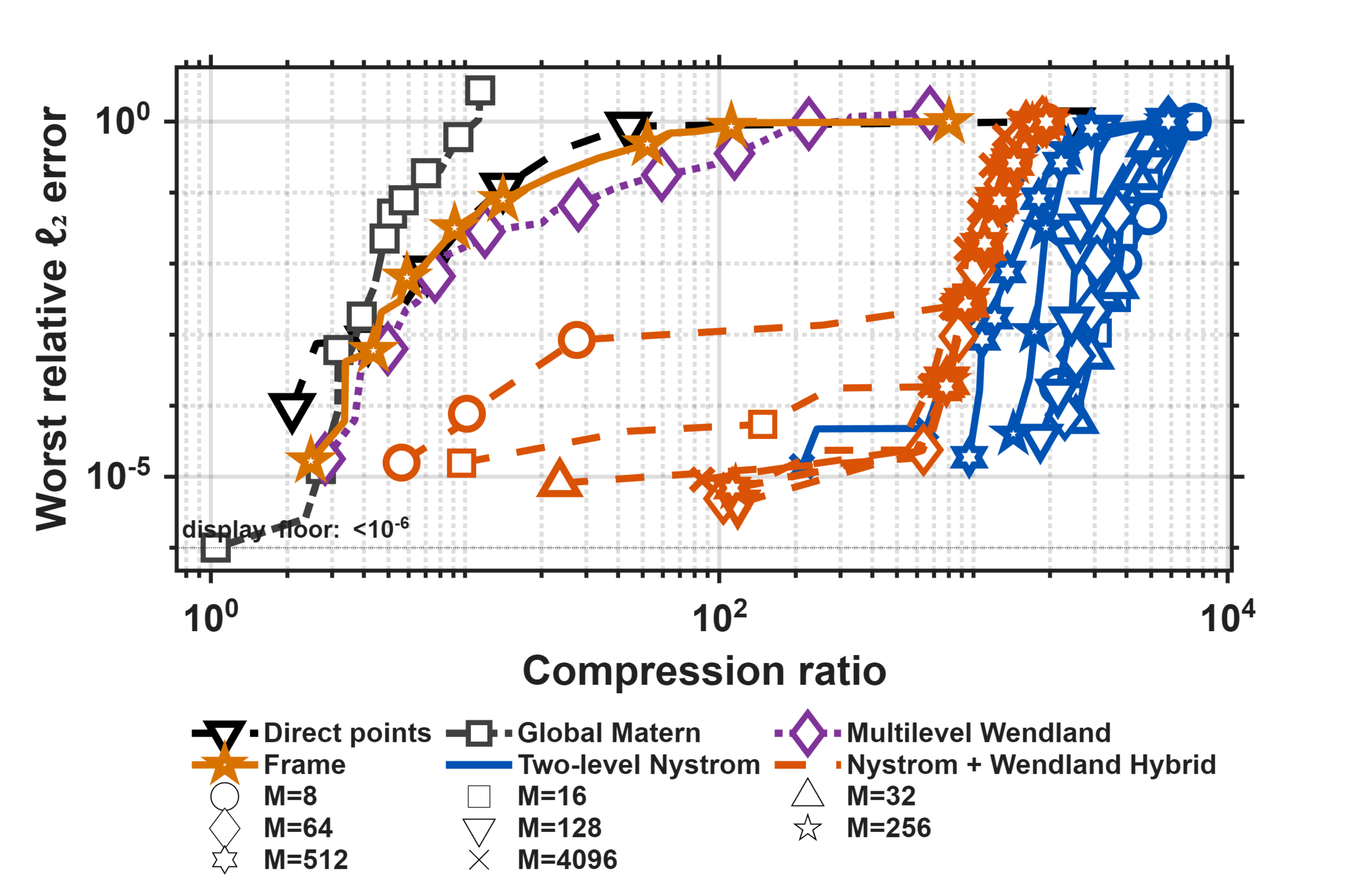}
 \caption{Periodic Taylor--Green flow. Rate--distortion curves for all tested
 landmark counts and comparison methods. Each curve is the Pareto envelope
 across four compressors, using the mean ratio $\overline R$ and maximum
 relative error $E_2$ over five fields from \eqref{eq:aggregate-rate-error}.
 Direct compression stores velocities
 at 7,288 distinct periodic nodes. Errors below $10^{-6}$ are placed on the labeled
 display floor.}
 \label{fig:rate-distortion-taylor}
\end{figure}

\paragraph{Species transport}
\Cref{fig:rate-distortion-species} shows a substantial approximation error for the tested \nyst{} spaces on the
species velocity fields: none meet a maximum relative error of $10^{-1}$.
The full Wendland correction removes this approximation error, but global Mat\'ern
already gives the best kernel compression at $10^{-3}$: $3.01\times$,
compared with $4.53\times$ for direct compression. The other interpolating
representations require slightly more storage. Here a hierarchy or a
redundant frame does not improve on the single global interpolant.
The hybrid becomes marginally more economical at $10^{-2}$, so even among
interpolants the preferred representation can depend on the error limit.
All kernel reconstructions remain divergence-free at every compressor
setting; these comparisons concern storage at matched velocity accuracy.

\begin{figure}[!htbp]
 \centering \includegraphics[page=3,width=0.80\linewidth]{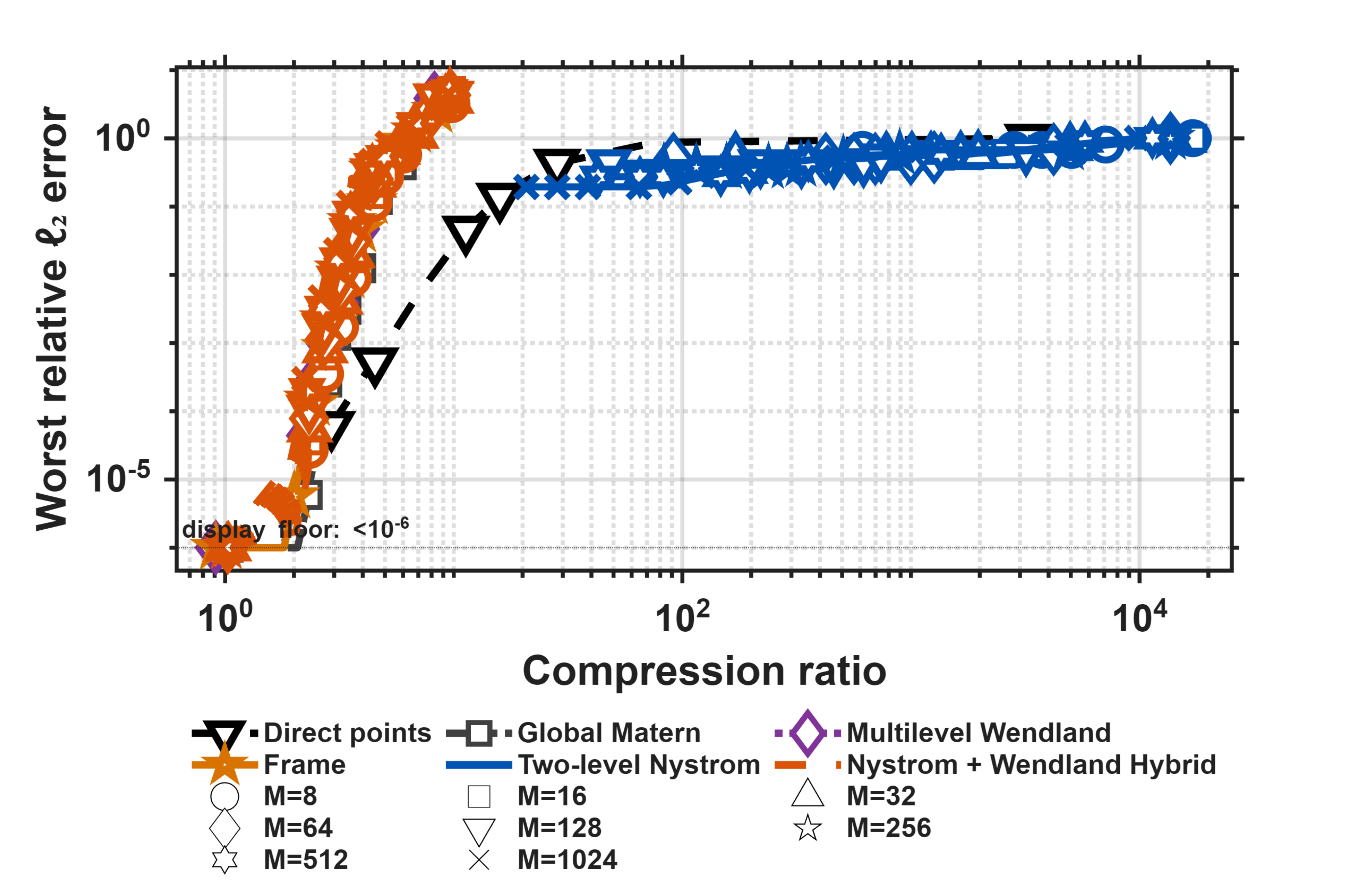}
 \caption{Three-dimensional species transport. The two-level \nyst{} family
 has a fitting error before compression; the full Wendland correction
 removes it to solver accuracy. Each curve is the Pareto envelope across
 four compressors, using the mean ratio and maximum relative error over
 five fields from \eqref{eq:aggregate-rate-error}. All tested landmark counts are shown.
 Errors below $10^{-6}$ are placed on
 the labeled display floor.}
 \label{fig:rate-distortion-species}
\end{figure}

%% file: results_property_preservation.tex
\subsection{Property preservation after compression}
\label{sec:property-results}
\label{sec:results-property-preservation}

For each method and landmark count, we select the most compressed setting
meeting each velocity-error limit on all five fields
(\cref{sec:results-quantities}), then examine its property errors.
Kernel fields remain divergence-free and, on periodic domains,
periodic at every compression tolerance. Changes in species
velocity statistics are measured at the selected settings.

The zero band in the divergence plots denotes analytic divergence freedom.
For direct compression, we estimate the divergence change using
\eqref{eq:divergence-measure}; the periodicity plots compare paired boundary
values. Property errors can increase when the selected compressor changes,
since selection maximizes storage savings subject to velocity accuracy.

We show each landmark count that meets at least one velocity-error limit.
Smaller spaces may exceed these limits even before compression.

\subsubsection{Incompressibility}

\paragraph{Lid-driven cavity}
Direct compression achieves a higher ratio at relative velocity error
$10^{-3}$, but introduces a nonzero divergence change. The frame gives the
highest compression among the tested divergence-free methods
(\cref{fig:divergence-lid}). The storage advantage of direct compression
therefore comes without the incompressibility guarantee.

\begin{figure}[tbp]
 \centering
 \includegraphics[page=4,width=0.8\linewidth]{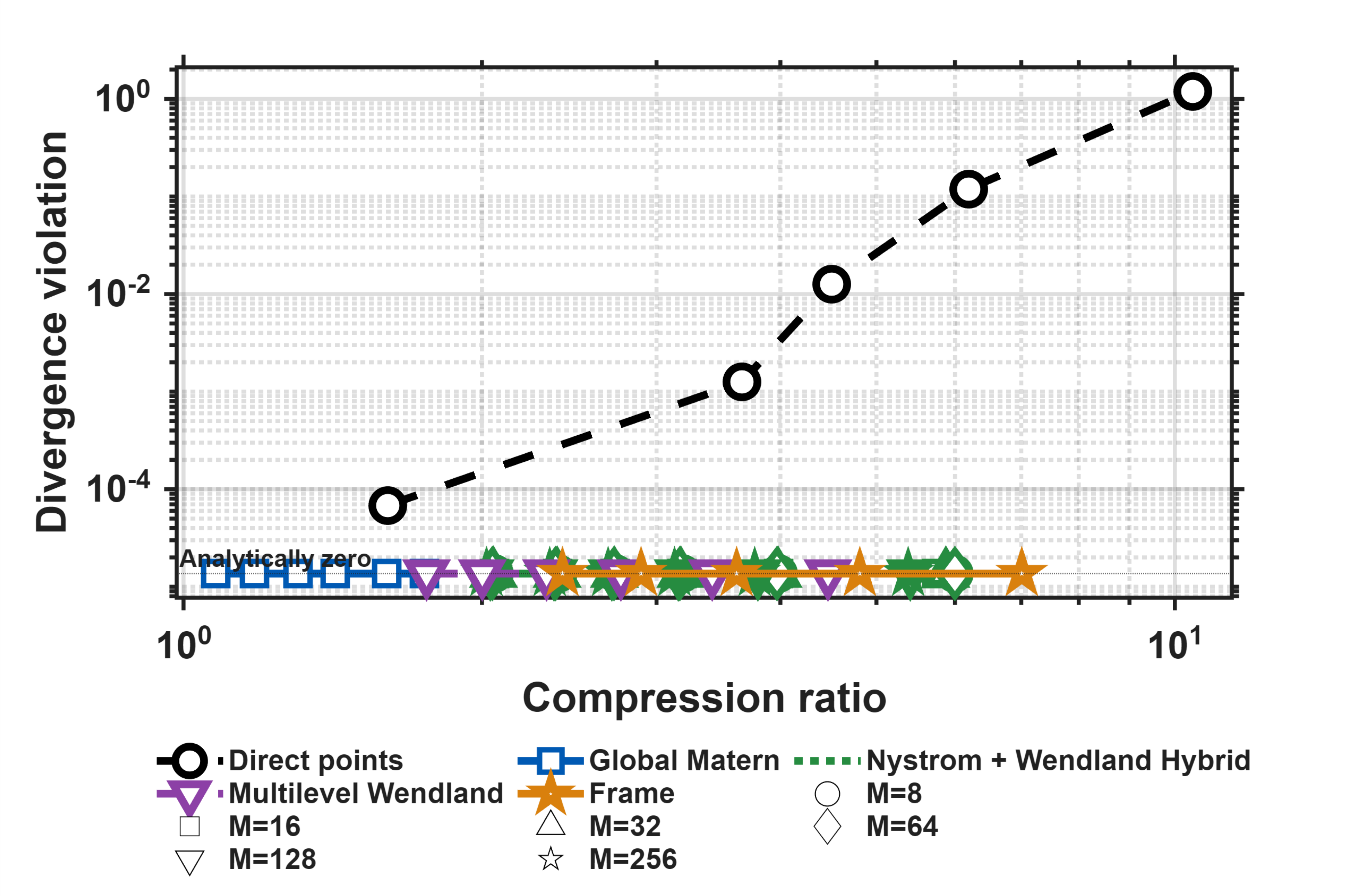}
 \caption{Lid-driven cavity. For each method and landmark count, each point
 uses the compressor setting giving the highest compression ratio while
 satisfying the relative-$\ell_2$ error limit on all five velocity fields.
 The zero band for kernel fields denotes analytic divergence freedom, not a
 measured PHS error. Direct points use the normalized local-PHS estimate of
 $\mathcal D_h(\widehat u-u)$ in \eqref{eq:divergence-measure}. Both coordinates
 are averaged over the five fields, with compression ratio defined in
 \eqref{eq:aggregate-rate-error}. We show each hybrid landmark count that
 meets an error limit; no tested two-level \nyst{} approximation meets these limits.}
 \label{fig:divergence-lid}
\end{figure}

\paragraph{Taylor--Green flow}
For Taylor--Green, two-level \nyst{} and the hybrid both exceed the
compression ratio of direct nodal storage at relative velocity error
$10^{-3}$ while retaining zero analytic divergence. Direct compression
introduces a nonzero divergence change at this error limit
(\cref{fig:divergence-taylor}).

\begin{figure}[tbp]
 \centering \includegraphics[page=5,width=0.8\linewidth]{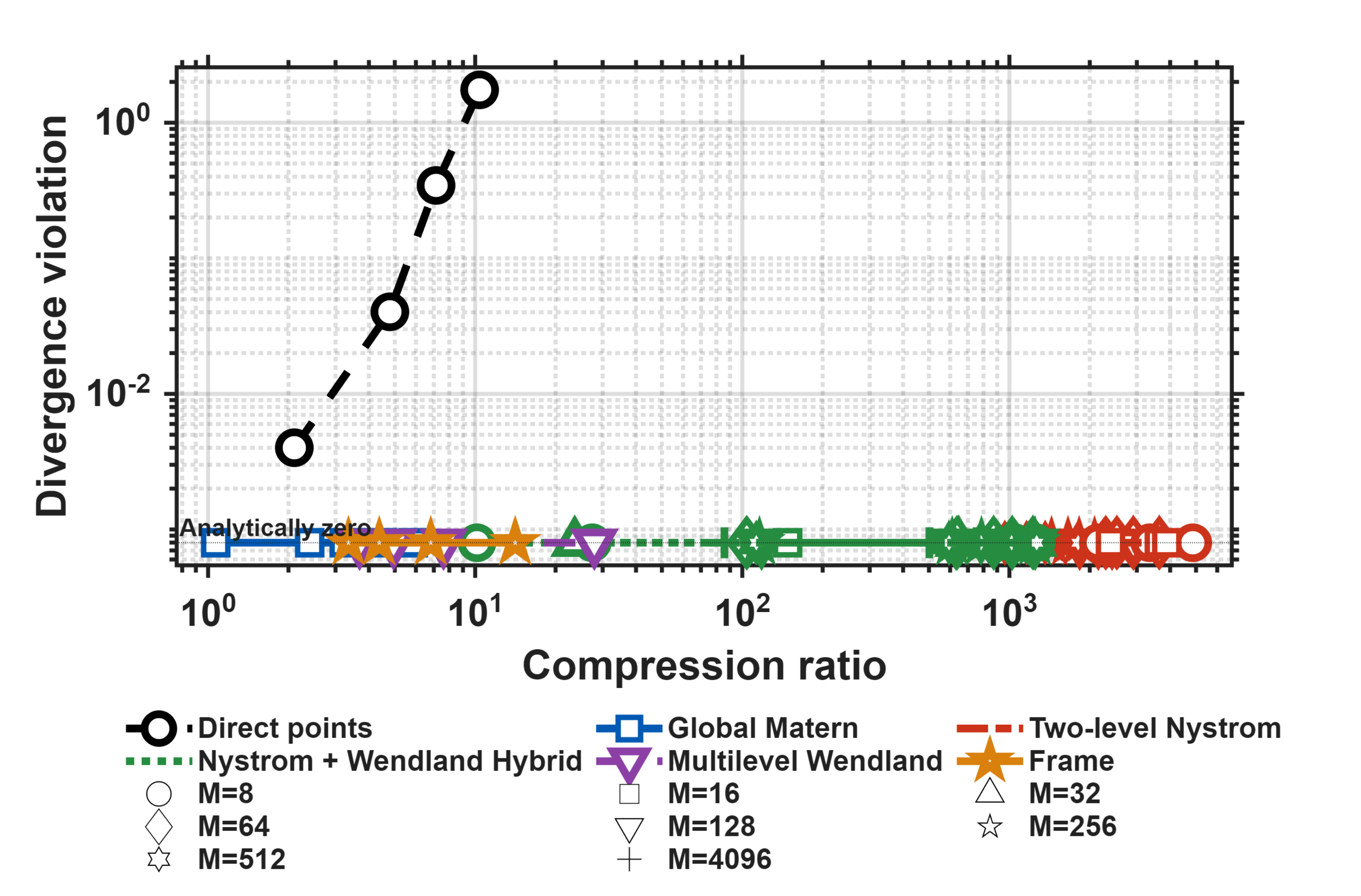}
 \caption{Taylor--Green divergence after compression. For each method and
 landmark count, we choose the most compressed result satisfying the same
 relative-$\ell_2$ error limit on every velocity field. The zero band denotes analytic
 divergence freedom of the kernel fields, not a measured PHS error.
 Direct points use the normalized local-PHS estimate of
 $\mathcal D_h(\widehat u-u)$ in \eqref{eq:divergence-measure}. Both coordinates
 are averaged over the five fields, with compression ratio defined in
 \eqref{eq:aggregate-rate-error}. We show all two-level and hybrid landmark
 counts that meet an error limit.}
 \label{fig:divergence-taylor}
\end{figure}

\paragraph{Species transport}
Direct nodal compression requires less storage at relative velocity error
$10^{-3}$, while global Mat\'ern is the most economical representation
that guarantees zero divergence (\cref{fig:divergence-species}). Direct
compression introduces a nonzero divergence change even at this small
velocity error.

\begin{figure}[tbp]
 \centering \includegraphics[page=7,width=0.8\linewidth]{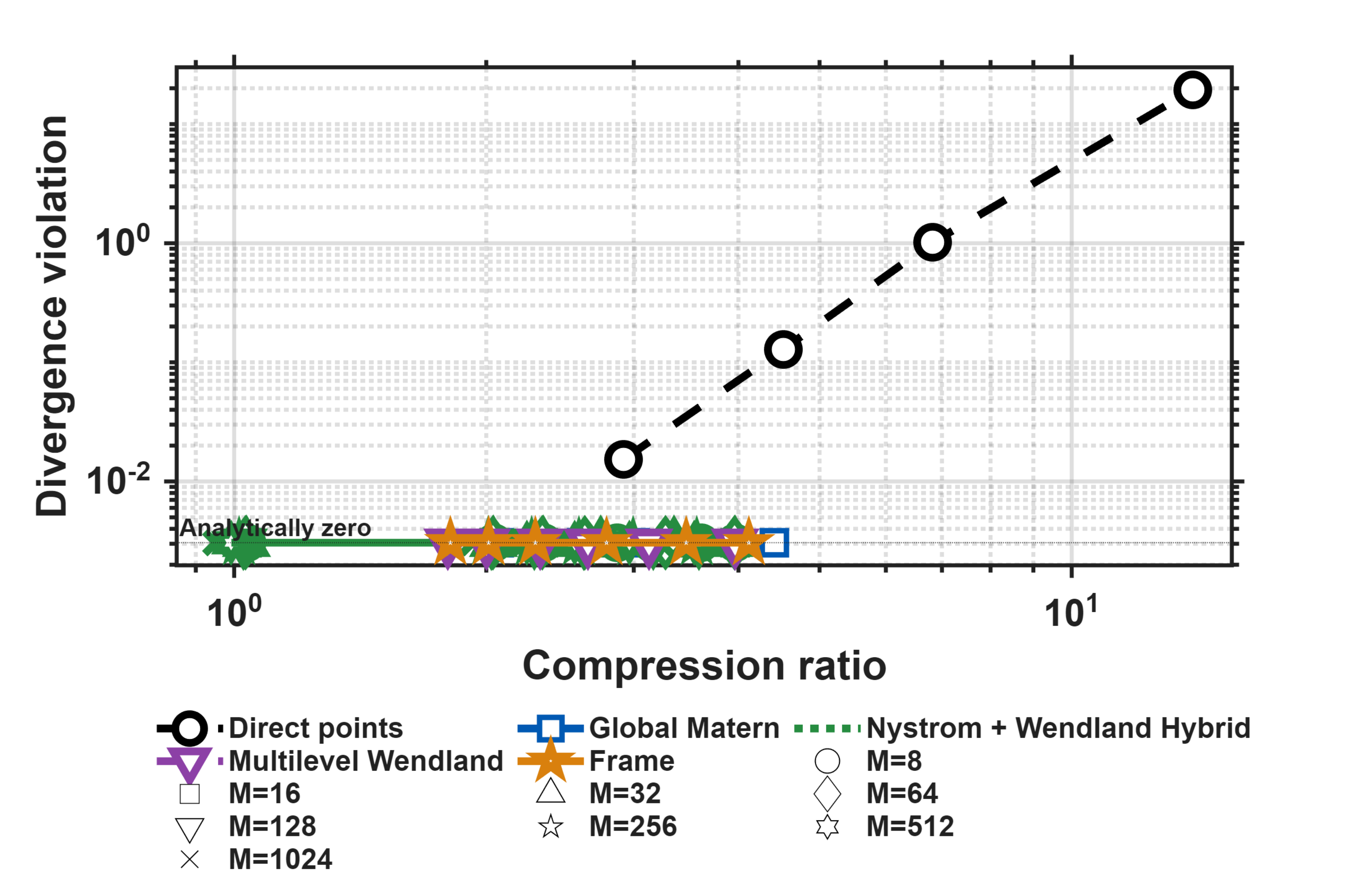}
 \caption{Species-transport divergence after compression. For each method
 and landmark count, we choose the most compressed result satisfying the
 same relative-$\ell_2$ error limit on every velocity field. The zero band denotes analytic
 divergence freedom of the kernel fields, not a measured PHS error.
 Direct points use the normalized local-PHS estimate of
 $\mathcal D_h(\widehat u-u)$ in \eqref{eq:divergence-measure}. Both coordinates
 are averaged over the five fields, with compression ratio defined in
 \eqref{eq:aggregate-rate-error}. We show each hybrid landmark count that
 meets an error limit; no tested two-level \nyst{} approximation meets these limits.}
 \label{fig:divergence-species}
\end{figure}

\subsubsection{Periodicity}

Direct compression need not assign identical reconstructed velocities to
the two copies of a periodic boundary node. Whether it does so depends on
the data and compressor setting. We compare this baseline with a simple
alternative that stores each periodic node once and copies its reconstructed
value to the matching boundary. This enforces nodal equality with little
change in compression ratio. The periodic kernel representations additionally
provide a periodic velocity field through their basis functions.
\Cref{fig:taylor-periodicity} reports the RMS mismatch defined in
\eqref{eq:periodicity-measures}; \cref{fig:taylor-periodicity-max} gives the maximum mismatch.
In both plots, the direct-point mismatch drops to zero at the selected PFPL
setting, then rises when the best compression ratio is obtained with MGARD.
The zero is specific to that PFPL setting, not a general periodicity guarantee.
We select for storage subject to a velocity-error limit, so a more compressed
setting can have a larger boundary mismatch. The upward jump therefore
reflects a change of compressor, rather than a threshold at which one
compressor begins to lose periodicity.

\begin{figure}[tbp]
 \centering
 \includegraphics[width=0.8\linewidth]{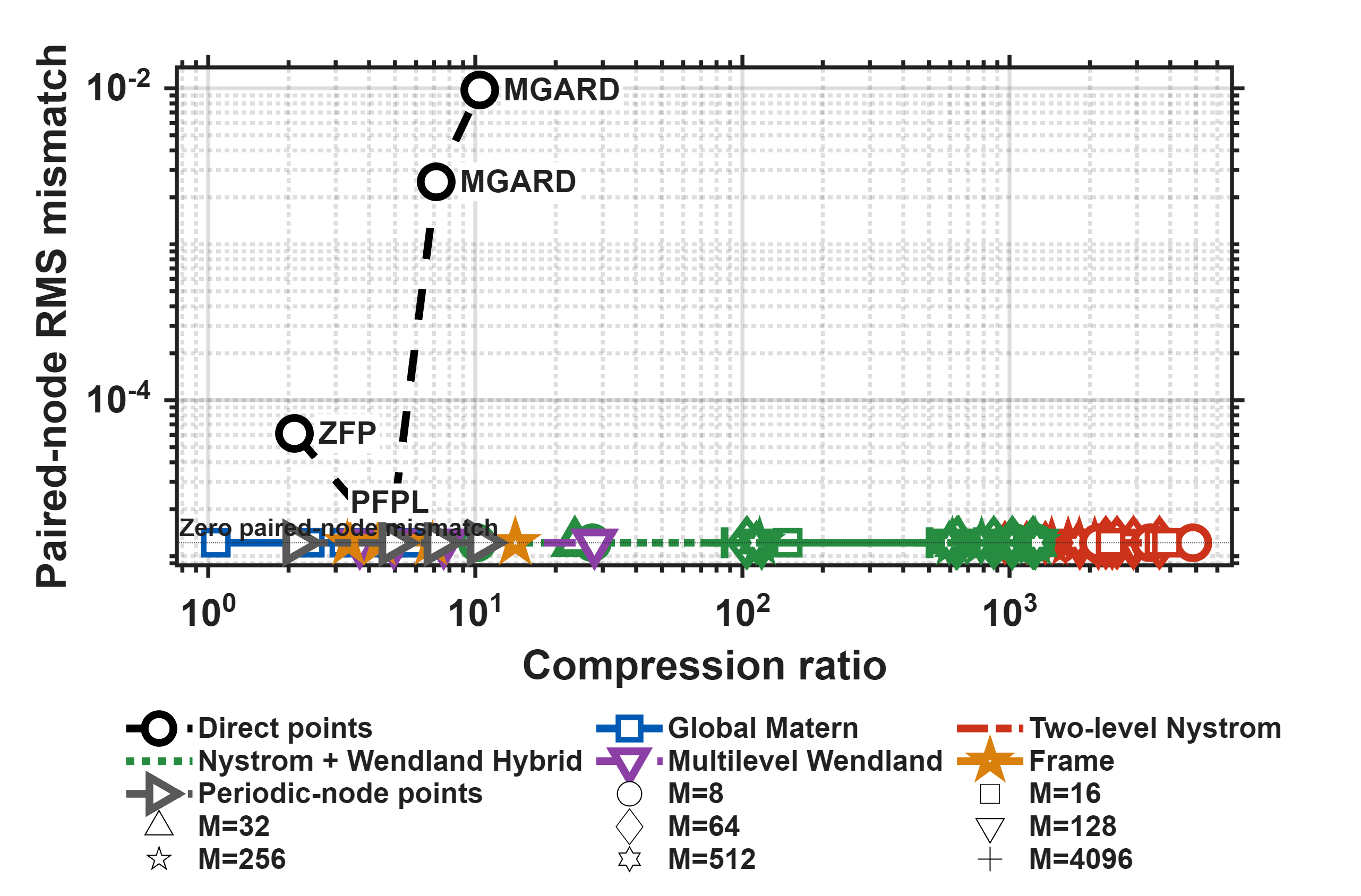}
 \caption{Taylor--Green normalized RMS jump over paired periodic boundaries, defined in \eqref{eq:periodicity-measures}, versus compression ratio from \eqref{eq:aggregate-rate-error}.
 Each point uses the most compressed result meeting the relative-$\ell_2$
 error limit on every velocity field; both
 plotted coordinates are averages over the five velocity fields. Labels identify the selected
 compressor for direct points (ZFP, PFPL, then MGARD at the two loosest
 error limits). The upward jump follows the switch from PFPL's zero boundary
 mismatch to MGARD's nonzero mismatch; selection maximizes compression
 subject to velocity accuracy, not periodicity. Exact zeros occupy the labeled
 zero-mismatch display band. Periodic-node points impose equality only on endpoint pairs;
 periodic kernels define a continuously periodic reconstructed field. We show
 all two-level and hybrid landmark counts that meet an error limit.}
 \label{fig:taylor-periodicity}
\end{figure}

\subsubsection{Velocity fluctuations across scales}
The physical-scale mixtures are designed to represent velocity variations
at several separations. We test how well their compressed reconstructions
retain the structure function of the source data. At relative velocity
error $10^{-3}$, direct compression, the frame, and
the hybrid give similarly small changes in the structure function
(\cref{fig:species-structure-function}), measured by \eqref{eq:structure-function-error}.

The uncompressed data have a mean fitted exponent of approximately $0.685$,
close to the Kolmogorov value $2/3$. At the same velocity-error limit,
the mean absolute exponent changes for direct points, global Mat\'ern,
multilevel Wendland, and the frame are of order $10^{-5}$.
These compression-induced changes are much smaller
than the variation caused by shifting or widening the fitting interval:
see \cref{sec:structure-window}. The agreement with $2/3$ is therefore specific
to the chosen interval. \Cref{fig:species-structure-slope} shows the
exponents as compression increases.

\begin{figure}[tbp]
 \centering
 \includegraphics[page=8,width=0.8\linewidth]{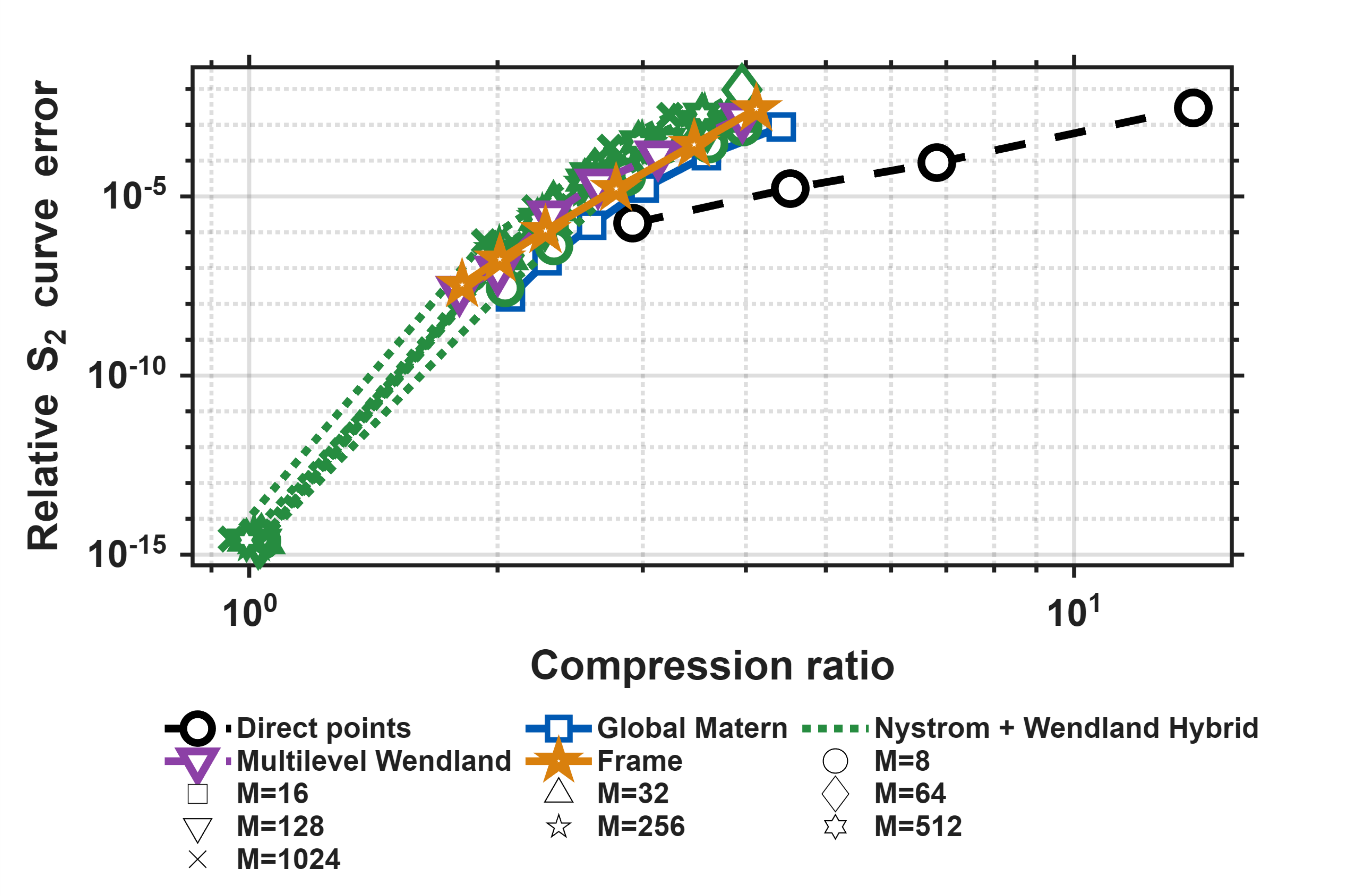}
 \caption{Species transport: structure-function error from
 \eqref{eq:structure-function-error} versus mean compression ratio $\overline R$
 from \eqref{eq:aggregate-rate-error}, averaged over the five velocity fields.
 For each method and landmark count, we select the most compressed result
 meeting the velocity-error limit on every field. We show each hybrid
 landmark count that meets an error limit.}
 \label{fig:species-structure-function}
\end{figure}

\begin{figure}[tbp]
 \centering
 \includegraphics[page=9,width=0.8\linewidth]{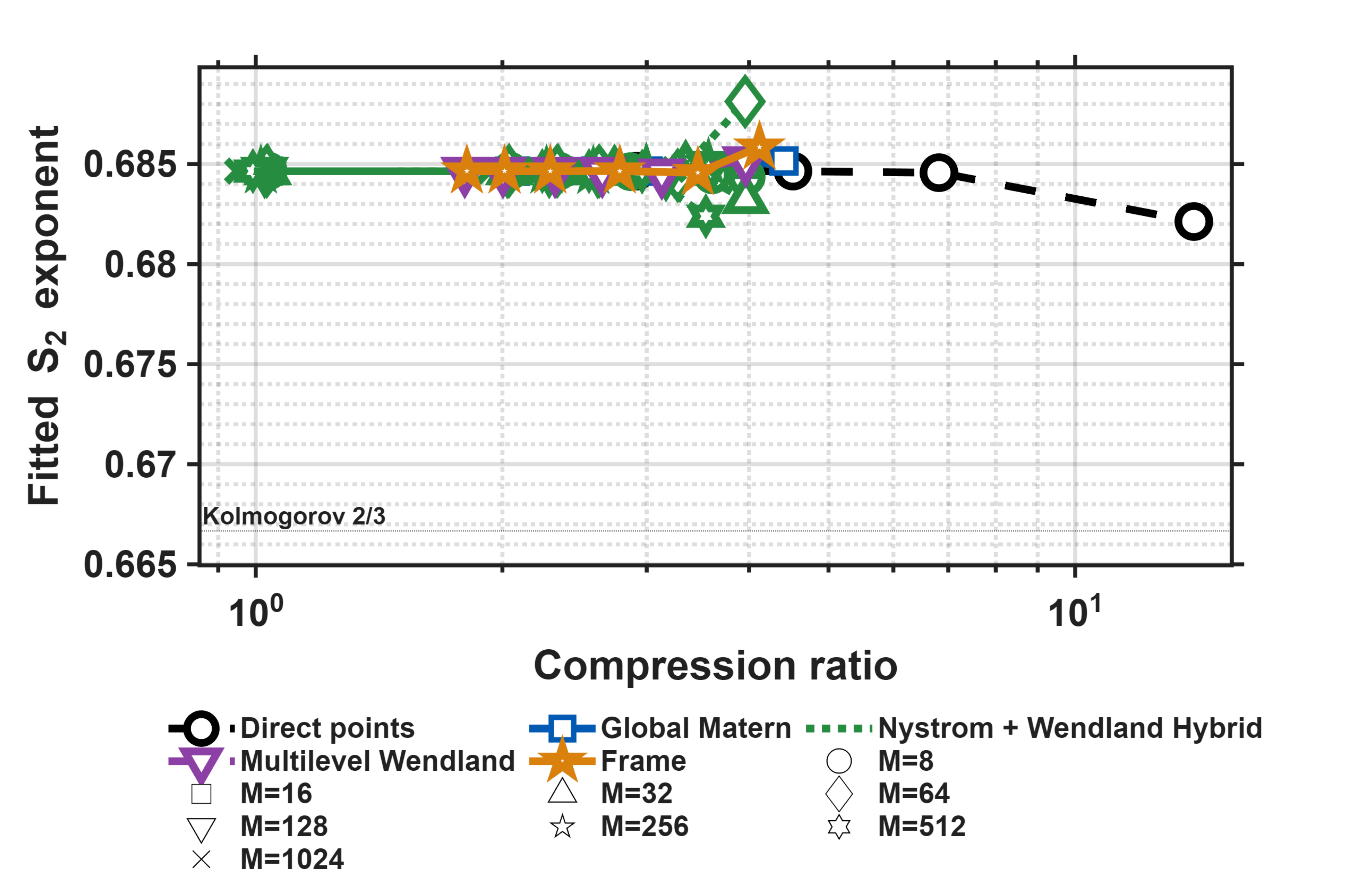}
 \caption{Species transport: fitted exponent $p$ versus compression ratio,
 both averaged over the five velocity fields. The quantities are defined
 in \eqref{eq:structure-exponent} and \eqref{eq:aggregate-rate-error}.
 We select the most compressed result meeting the velocity-error limit on
 every field and show each hybrid landmark count that meets an error limit.
 The horizontal reference marks the Kolmogorov value $2/3$.}
\label{fig:species-structure-slope}
\end{figure}

\subsection{Residual levels and separate compression of frame blocks}
\label{sec:guarantees-ablations}
\label{sec:results-guarantee}

The hybrid and frame offer further choices in how coefficients are fitted
and stored. We examine whether an additional residual level improves the
hybrid and whether ordering and compressing frame blocks separately
improves their storage.

\paragraph{An additional residual level}
One all-node Wendland correction already removes the nodal projection
error. An intermediate correction changes the distribution of that
residual among the coefficient blocks. We tested whether this improves
compression using two Wendland levels. The cavity comparison
favored one correction. For Taylor--Green
at relative error $10^{-3}$, we find the largest compression ratio meeting
the error limit separately for each velocity field and compressor. We then compare
the two-correction ratio with the one-correction ratio on the same velocity field.
Across the five velocity fields, the median reduction in compression ratio is approximately
$2$--$18\%$ for the CPU codecs and $1$--$20\%$ for the GPU codecs.
We retain one correction on this evidence. Here the settings are selected
separately for each velocity field; the main figures use one shared setting and
the mean compression ratio.

\paragraph{Frame ordering and separate block compression}
Grouping and reordering frame coefficients leave the uncompressed fit
unchanged, but can affect compression. For the Taylor--Green
frame, separating the three coefficient blocks without
reordering generally gives no benefit at the same $10^{-3}$ error limit.
Grouping nearby centers before splitting the blocks improves compression
for all four CPU codecs, but not for all GPU codecs. We group the centers
within each block in Morton order, which interleaves the bits of their
spatial coordinates. Relative to single-array compression, the median
gains are approximately $9\%$ for ZFP, $23\%$ for PFPL, $34\%$ for SZ3,
and $17\%$ for MGARD. The main CPU frame curves use this ordering.
On the GPU, PFPL gains the same $23\%$, CUDA ZFP changes little, and
cuSZ and MGARD-X lose approximately $20\%$ and $22\%$, respectively.
These comparisons include the bytes needed to store the three stream
lengths and hold the coefficient tolerances fixed.

We also compressed the three species frame levels as separate arrays, keeping
their within-level ordering and scalar tolerances unchanged. Including the
stream-length overhead, the best ratio at $10^{-2}$ changed from $3.464$
to $3.462$ on both CPU and GPU. We therefore retain single-array compression.
Error allocation among levels remains to be optimized.

\input{results_gpu_native}

\subsection{Limitations}
\label{sec:limitations}
Our comparisons include choices of kernel scales and centers as well as
fitting rules. In particular,
the multilevel Wendland and frame constructions use different supports
as well as sequential and simultaneous fits. We evaluate five velocity fields per
dataset. Full trajectories, long-time restart
stability, and downstream solver accuracy require further study. Compression
ratios exclude geometry, center indices, and kernel parameters shared on a
fixed mesh. The CPU and GPU experiments do not include a consistent set of
execution-time measurements, so they establish no speed advantage.
The property tests cover incompressibility, periodicity, and selected
second-order velocity statistics. Automatic choice of landmark counts,
center hierarchies, and kernel scales from an accuracy and storage budget
remains an open problem.

%% file: results_gpu_native.tex
\subsection{GPU compressor comparison}
\label{sec:results-gpu-native}

Our kernel fitting and evaluation run on the GPU. Keeping the fitted
coefficients in GPU memory avoids a transfer to the host for compression.
We therefore repeat the comparison with CUDA ZFP, PFPL-CUDA, cuSZ, and
MGARD-X~\cite{ZFPDocumentation,FallinEtAl2025,TianEtAl2020,GongEtAl2024},
using the same property-preserving representations. We compare
compressed sizes at matched reconstruction errors because several GPU
compressors use different error controls from their CPU counterparts.
Details of GPU evaluation and the remaining host work appear in
\cref{sec:implementation}.

PFPL gives the same rate--distortion curve on both devices in
\cref{fig:gpu-native-species}, consistent with its design to produce
identical compressed files~\cite{FallinEtAl2025}. We use its absolute-error
mode in both experiments.
The remaining CPU and GPU codecs use different implementations or
error-control modes. CUDA ZFP uses a
fixed rate, prescribing bits per value rather than an error tolerance,
and reaches at most approximately $16\times$ in the tested range.
CPU ZFP instead prescribes an absolute error tolerance.
MGARD-X satisfies its RMS bounds and achieves lower compression ratios
than CPU MGARD at comparable field errors on these one-dimensional
coefficient vectors. Cavity and Taylor--Green results appear
in \cref{sec:gpu-codec-appendix}.

For cuSZ, we check the reconstructed coefficients on the device before
accepting the compressed stream. Small arrays are stored uncompressed;
larger arrays also revert to uncompressed storage if their maximum
coefficient error exceeds $1.1$ times the requested tolerance. This acceptance
threshold is looser than the requested bound and must be used when
applying a coefficient-error estimate to accepted cuSZ results.
We did not measure fitting, compression, decompression, and evaluation
times consistently across the CPU and GPU runs, so we do not report
CPU-to-GPU speedups.

\begin{figure}[tbp]
  \centering
  \includegraphics[page=11,width=0.72\linewidth]{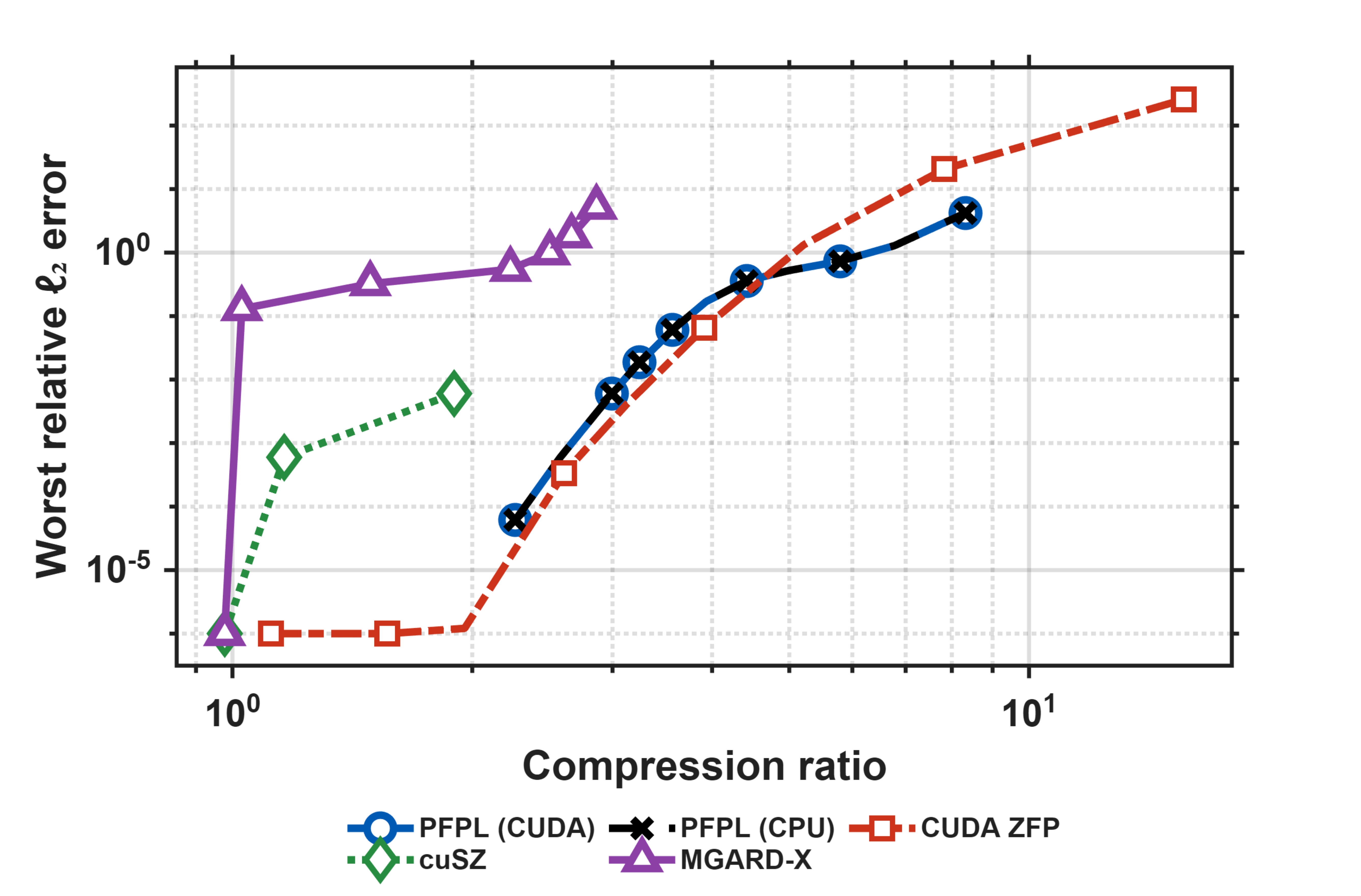}
  \caption{Device-resident compression of the \nyst{}--Wendland hybrid for
  species transport with $M=256$. Each point reports the mean compression ratio and
  largest relative $\ell_2$ error over five velocity fields, as defined in
  \eqref{eq:aggregate-rate-error}.
  Values below $10^{-6}$ are drawn at $10^{-6}$. Curves retain nondominated
  settings for each compressor; PFPL (CPU) is the reference curve.
  CUDA ZFP uses fixed rate; the other modes control error.}
  \label{fig:gpu-native-species}
\end{figure}

%% file: conclusion.tex
\section{Conclusions}
\label{sec:conclusions}
We presented several methods for property-preserving compression of incompressible and periodic flow fields. Our methods all preserved these properties analytically, were accompanied by nodal error bounds, and allowed lossy compression without relaxing incompressibility or periodicity.

Our results indicate that the choice of representation depends first on the required accuracy. In particular, we found that if an uncompressed two-level \nyst{} fit is sufficiently accurate, its
small coefficient count can yield substantial storage savings. If its fitting error exceeds the target, tighter
compression tolerances cannot help; residual interpolation or another
interpolating representation is needed. Among these representations,
compressed size depends on the flow and compressor, not just the number
of coefficients. We therefore recommend comparing a few representations (perhaps one from each broad class)
at the intended accuracy on representative velocity fields. GPU
compressors also avoid host transfers of the coefficients, but their
different compression modes require comparisons at matched field errors.

Direct nodal compression remains competitive when exact constraints are
unnecessary. Periodic-node storage enforces equality at paired boundary
nodes, whereas periodic kernels enforce it throughout the domain and
also preserve zero divergence. Neither guarantee ensures preservation of
turbulent velocity statistics: their changes must be checked over the
spatial scales of interest.

%% file: ex_supplement_body.tex
\hypersetup{
  pdftitle={Supplementary Materials: Property-Preserving Data Compression for Incompressible Flows},
  pdfauthor={Andrey Pluzhnik, Ramansh Sharma, Martin Burtscher, and Varun Shankar}
}

\begin{center}
\bfseries\MakeUppercase{Supplementary Materials: Property-Preserving Data
Compression for Incompressible Flows}
\end{center}

\vspace{0.11in}

These supplementary materials give the experimental parameters,
additional error comparisons, and results for individual compressors.
References to sections, equations, figures, and tables without an SM prefix
refer to the accompanying article.

\input{appendix_codec}

%% file: appendix_codec.tex
\FloatBarrier
\input{experiments}
\input{bound_details}
\FloatBarrier
\raggedbottom
\section{Maximum pointwise errors and comparisons by compressor}
\label{sec:codec-appendix}

The main figures combine compressors and use a relative Euclidean error.
Here we examine the largest local velocity errors and the performance of
each compressor separately. The pointwise and individual-compressor curves use
128 landmarks per Mat\'ern level for cavity and 256 for
Taylor--Green and species transport; the main figures show every tested
landmark count. The pointwise-error curves report $E_\infty$ from
\eqref{eq:maximum-pointwise-error}; the curves for individual compressors
report the largest relative error $E_2$ defined in \cref{sec:results-quantities}.
Settings that fail the error checks, dominated points, and
duplicates are omitted. Plots at prescribed error limits also omit methods
that cannot meet those limits. Line style and marker shape distinguish methods
independently of color.

The periodic-boundary comparison in \cref{fig:taylor-periodicity-max}
uses the same cross-compressor selection as the RMS comparison in
\cref{fig:taylor-periodicity}. The direct-point mismatch rises when the
selected compressor changes from PFPL, which gives matching boundary values
at that setting, to MGARD, which does not. This occurs because we maximize
compression subject to velocity accuracy, without imposing a boundary-mismatch
limit. It is not a loss-of-periodicity threshold for a single compressor.

\newcommand{\singlecodecfigure}[3]{%
\begin{figure}[!htbp]\centering
\includegraphics[page=#1,width=.88\linewidth]{plot_bank_corrected_20260929.pdf}
\caption{#2}\label{#3}
\end{figure}}

\newcommand{\compactcodecfigure}[3]{%
\begin{figure}[!htbp]\centering
\includegraphics[page=#1,width=.88\linewidth]{plot_bank_corrected_20260929.pdf}
\caption{#2}\label{#3}
\end{figure}}

\begin{figure}[!htbp]
\centering
\includegraphics[page=12,width=.88\linewidth]{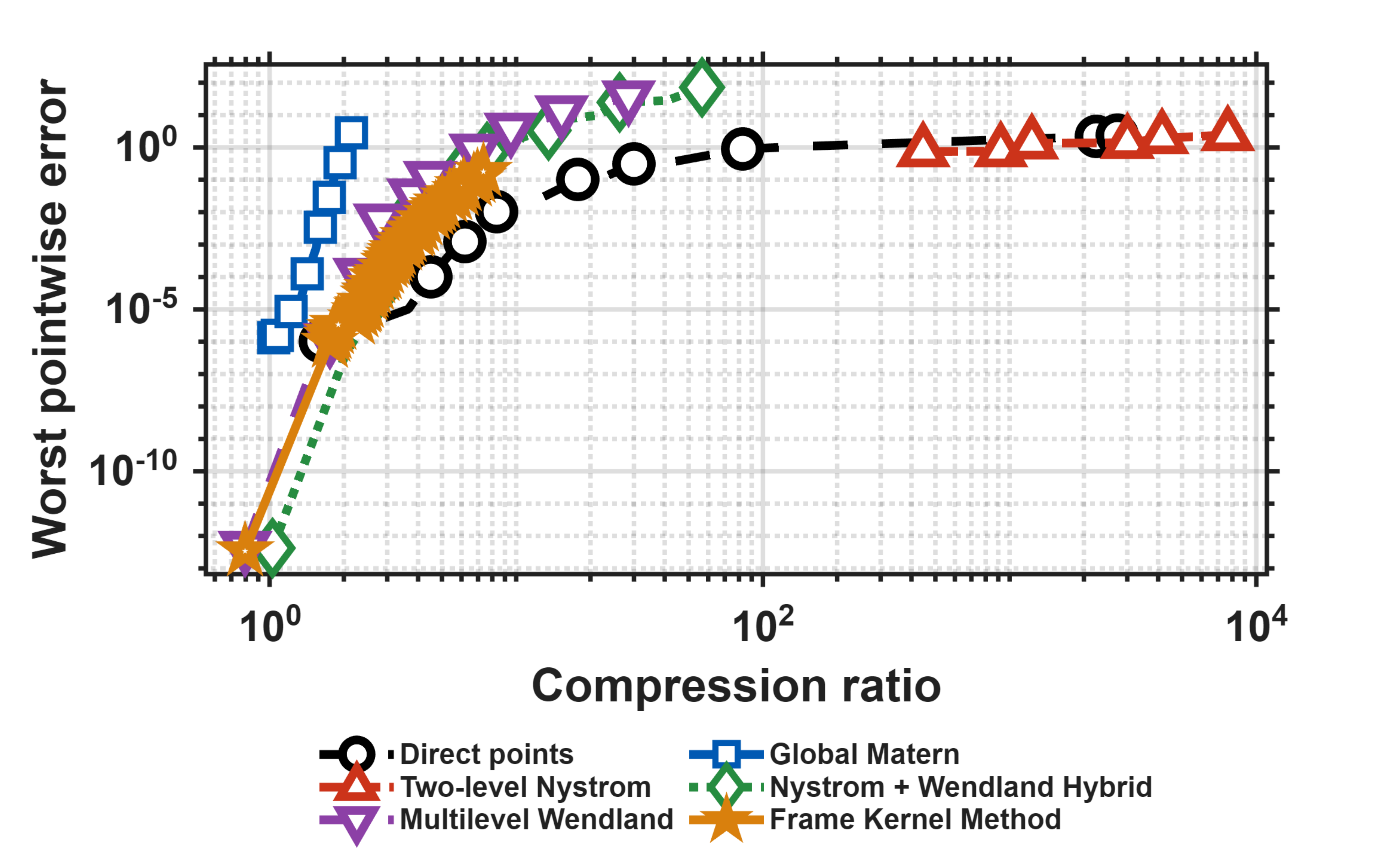}
\caption{Lid-driven cavity with $M=128$: largest absolute pointwise velocity error over all nodes
and five velocity fields versus mean compression ratio; lines are nondominated envelopes
across the four compressors. The quantities are defined in \cref{eq:maximum-pointwise-error,eq:aggregate-rate-error}.}
\label{fig:pointwise-lid}
\end{figure}
\singlecodecfigure{13}%
{Taylor--Green with $M=256$: largest absolute pointwise velocity error over all nodes and five velocity fields versus
mean compression ratio; lines are nondominated envelopes across the four
compressors. See \cref{eq:maximum-pointwise-error,eq:aggregate-rate-error} for the definitions.}%
{fig:pointwise-taylor}
\singlecodecfigure{14}%
{Species transport with $M=256$: largest absolute pointwise velocity error over all nodes and five velocity fields
versus mean compression ratio; lines are nondominated envelopes across
the four compressors. See \cref{eq:maximum-pointwise-error,eq:aggregate-rate-error} for the definitions.}%
{fig:pointwise-species}

\singlecodecfigure{15}%
{Lid-driven cavity with ZFP: largest relative $\ell_2$ error over five velocity fields versus mean
compression ratio; lines retain nondominated settings. Both quantities are defined in \eqref{eq:aggregate-rate-error}.}{fig:codec-lid-zfp}
\singlecodecfigure{16}%
{Lid-driven cavity with PFPL: largest relative $\ell_2$ error over five velocity fields versus mean
compression ratio; lines retain nondominated settings. Both quantities are defined in \eqref{eq:aggregate-rate-error}.}{fig:codec-lid-pfpl}
\singlecodecfigure{17}%
{Lid-driven cavity with SZ3: largest relative $\ell_2$ error over five velocity fields versus mean
compression ratio; lines retain nondominated settings. Both quantities are defined in \eqref{eq:aggregate-rate-error}.}{fig:codec-lid-sz3}
\singlecodecfigure{18}%
{Lid-driven cavity with MGARD: largest relative $\ell_2$ error over five velocity fields versus
mean compression ratio; lines retain nondominated settings. Both quantities are defined in \eqref{eq:aggregate-rate-error}.}{fig:codec-lid-mgard}

\singlecodecfigure{19}%
{Taylor--Green with ZFP: largest relative $\ell_2$ error over five velocity fields versus mean
compression ratio; lines retain nondominated settings. Both quantities are defined in \eqref{eq:aggregate-rate-error}.}{fig:codec-taylor-zfp}
\singlecodecfigure{20}%
{Taylor--Green with PFPL: largest relative $\ell_2$ error over five velocity fields versus mean
compression ratio; lines retain nondominated settings. Both quantities are defined in \eqref{eq:aggregate-rate-error}.}{fig:codec-taylor-pfpl}
\singlecodecfigure{21}%
{Taylor--Green with SZ3: largest relative $\ell_2$ error over five velocity fields versus mean
compression ratio; lines retain nondominated settings. Both quantities are defined in \eqref{eq:aggregate-rate-error}.}{fig:codec-taylor-sz3}
\singlecodecfigure{22}%
{Taylor--Green with MGARD: largest relative $\ell_2$ error over five velocity fields versus mean
compression ratio; lines retain nondominated settings. Both quantities are defined in \eqref{eq:aggregate-rate-error}.}{fig:codec-taylor-mgard}
\singlecodecfigure{23}%
{Taylor--Green: maximum normalized boundary mismatch from
\eqref{eq:periodicity-measures} versus compression ratio from
\eqref{eq:aggregate-rate-error}. Both quantities are averaged
over the five velocity fields. For each method and landmark count, we choose
the most compressed result meeting the velocity-error limit on every field.
The direct-point curve switches from ZFP to PFPL and then MGARD as the
error limit increases. Its upward jump occurs when
MGARD replaces a PFPL setting with zero boundary mismatch.
All feasible landmark counts are shown; exact zeros occupy the labeled
zero-mismatch band.}%
{fig:taylor-periodicity-max}

\singlecodecfigure{24}%
{Species transport with ZFP: largest relative $\ell_2$ error over five velocity fields versus mean
compression ratio; lines retain nondominated settings. Both quantities are defined in \eqref{eq:aggregate-rate-error}.}{fig:codec-species-zfp}
\compactcodecfigure{25}%
{Species transport with PFPL: largest relative $\ell_2$ error over five velocity fields versus mean
compression ratio; lines retain nondominated settings. Both quantities are defined in \eqref{eq:aggregate-rate-error}. The direct-point curve
flattens near unit relative error as loose tolerances erase the velocity
variation. At the loosest settings, all five velocity fields have unit relative error
and compress to 84 bytes each (approximately $3254\times$). This high-ratio
tail represents loss of the signal, not accurate reconstruction. The compressor
is fixed throughout this figure.}{fig:codec-species-pfpl}
\compactcodecfigure{26}%
{Species transport with SZ3: largest relative $\ell_2$ error over five velocity fields versus mean
compression ratio; lines retain nondominated settings. Both quantities are defined in \eqref{eq:aggregate-rate-error}.}{fig:codec-species-sz3}

\begin{figure}[!htbp]
\centering
\includegraphics[page=27,width=.88\linewidth]{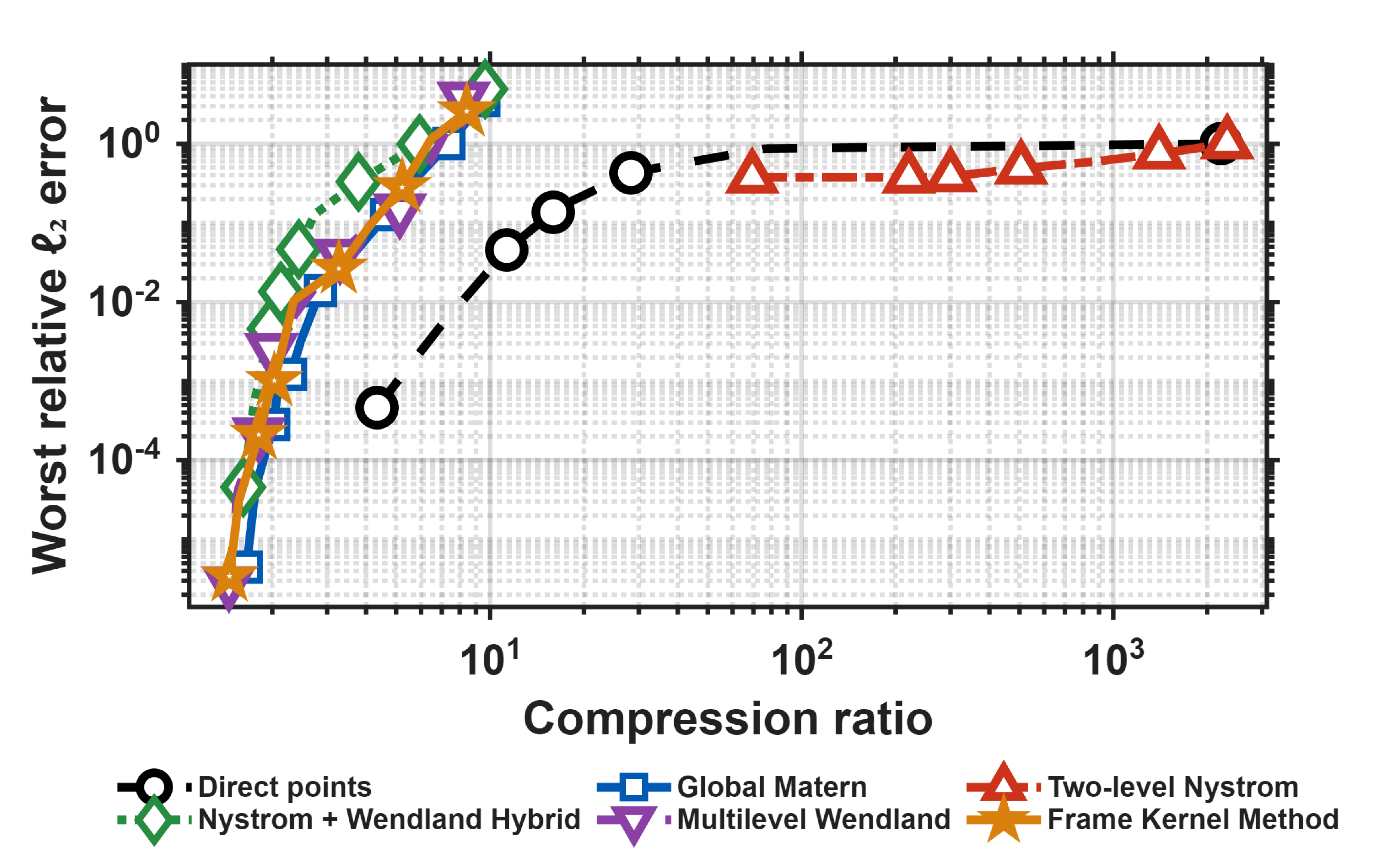}
\caption{Species transport with MGARD: largest relative $\ell_2$
error over five velocity fields versus mean compression ratio; lines retain nondominated settings. Both quantities are defined in \eqref{eq:aggregate-rate-error}.}
\label{fig:codec-species-mgard}
\end{figure}

\FloatBarrier

\clearpage
\section{Device-resident compressor curves}
\label{sec:gpu-codec-appendix}

\Cref{fig:gpu-native-lid,fig:gpu-native-taylor} extend the GPU comparison
in \cref{sec:results-gpu-native} to cavity and Taylor--Green flow.
The CPU and GPU compression modes differ as described there; the curves
compare compressed size at measured reconstruction errors.

\begin{figure}[H]
\centering
\includegraphics[page=28,width=.88\linewidth]{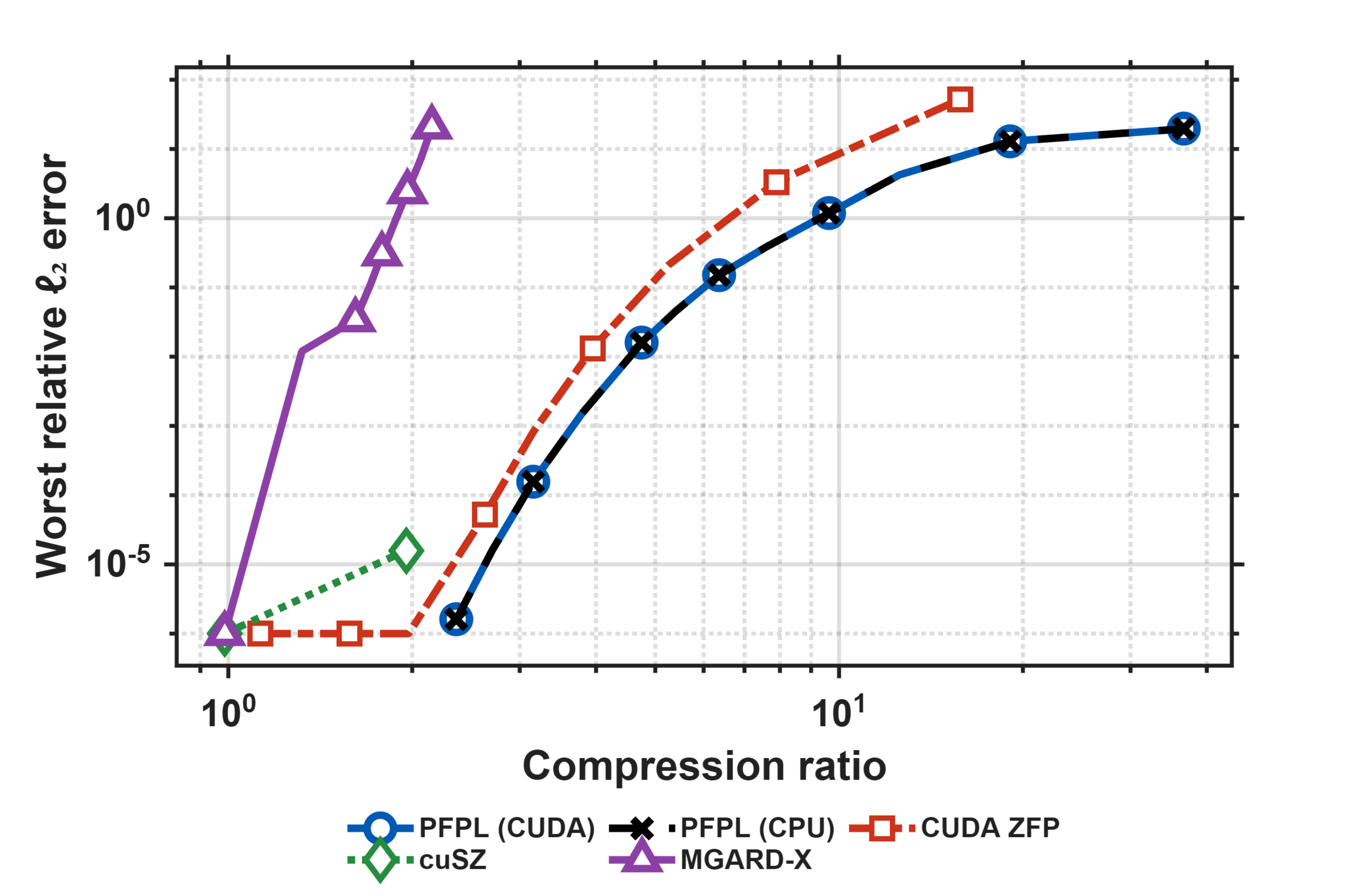}
\caption{Device-resident compressor comparison for the \nyst{}--Wendland hybrid on
lid-driven cavity flow with $M=128$: largest relative $\ell_2$ error over five velocity fields versus
mean compression ratio, as defined in \eqref{eq:aggregate-rate-error}.
Curves retain nondominated settings for each compressor, with PFPL (CPU)
as reference. CUDA ZFP uses fixed rate; the other modes control error.
Values below $10^{-6}$ are drawn at $10^{-6}$.}
\label{fig:gpu-native-lid}
\end{figure}

\begin{figure}[H]
\centering
\includegraphics[page=29,width=.88\linewidth]{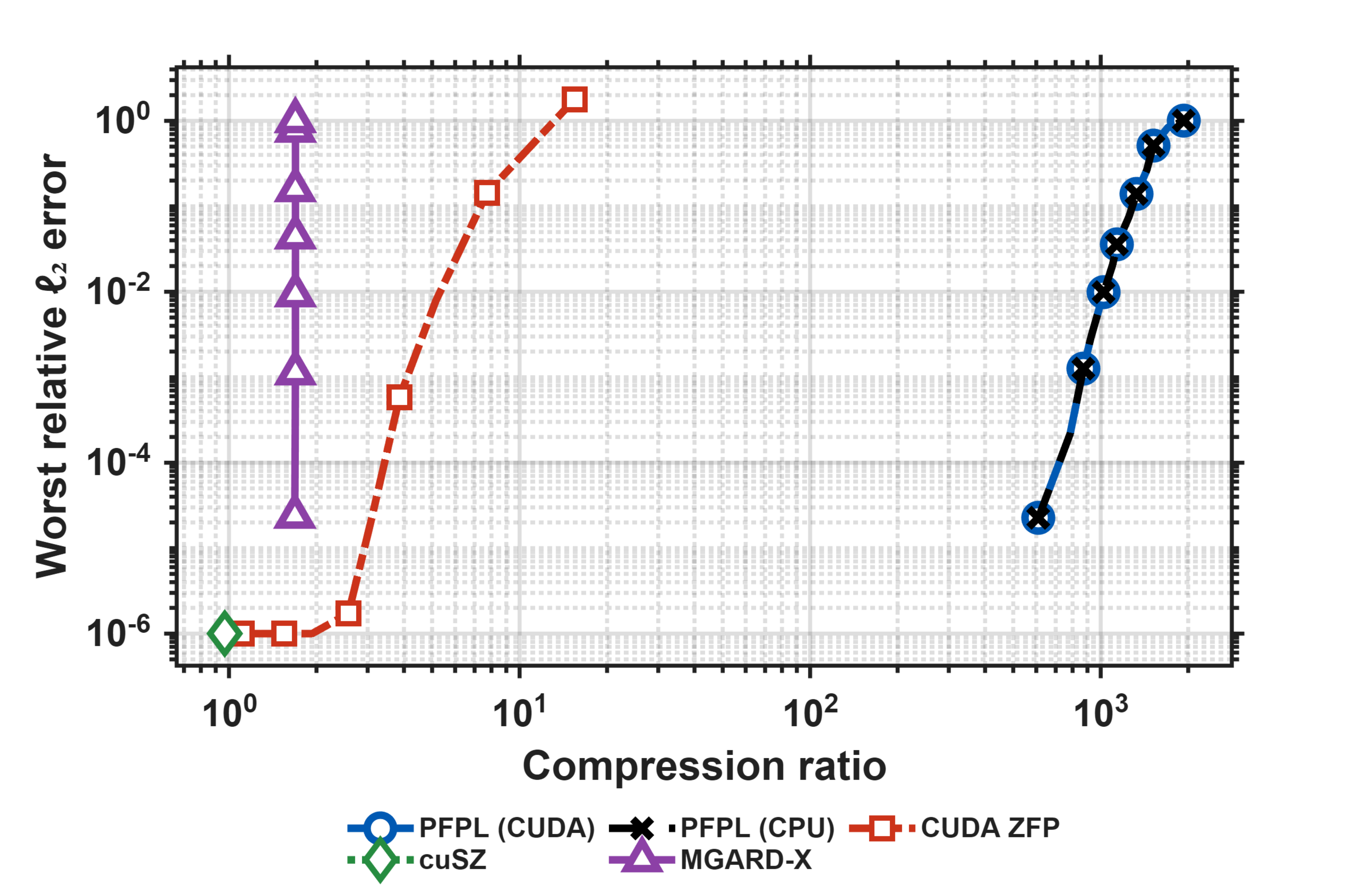}
\caption{Device-resident compressor comparison for the periodic
\nyst{}--Wendland hybrid on Taylor--Green with $M=256$: largest relative $\ell_2$ error over
five velocity fields versus mean compression ratio, as defined in \eqref{eq:aggregate-rate-error}.
Curves retain nondominated settings for each compressor, with PFPL (CPU)
as reference. CUDA ZFP uses fixed rate; the other modes control error.
Values below $10^{-6}$ are drawn at
$10^{-6}$.}
\label{fig:gpu-native-taylor}
\end{figure}

\FloatBarrier

%% file: experiments.tex
\Needspace{8\baselineskip}
\section{Reproducibility and experimental parameters}
\label{sec:experiments}

\subsection{Datasets and required structure}

\begin{table}[t]
\caption{Benchmarks, velocity dimension $d$, active-node counts after periodic
deduplication where applicable, and properties of the kernel reconstructions.}
\label{tab:datasets}
\centering
\begin{tabular}{@{}lcrp{0.39\linewidth}@{}}
\toprule
Dataset & $d$ & Active nodes & Reconstruction properties \\
\midrule
Lid-driven cavity & 2 & 9,566 & divergence-free \\
Taylor--Green & 2 & 7,288 & divergence-free and periodic \\
Species transport & 3 & 11,388 & divergence-free; physical-length kernel mixture \\
\bottomrule
\end{tabular}
\end{table}

\Cref{tab:datasets} lists the datasets and kernel constraints.
The cavity data are velocity fields on a 9,566-node triangular mesh of
$[0,1]^2$. The flow is generated by a horizontally moving top wall, with the
remaining walls no-slip, and is sampled at the terminal time $T=5$; the
realizations vary the constant lid speed~\cite{FriedenEtAl2026}. This benchmark
contains wall layers and corner structure that are not well represented by a
small global landmark space.

The Taylor--Green dataset contains four ensembles, stored at the reported times 7, 8, 9,
and 10, on $[0,2\pi]^2$. We pool these ensembles before selecting fields for testing.
The archive has 7,447 nodes; identifying the two periodic coordinate
directions leaves 7,288 unique nodes and 159 endpoint copies. The main direct
rate--distortion baseline stores one velocity per distinct periodic node. The property experiment also
compresses all 7,447 nodal velocities independently to measure the periodicity defect of
unconstrained nodal compression.

The species data are three-component velocity fields in a static three-blade
Kenics mixer. Air and methane enter opposite halves of the inlet, the solid
surfaces are no-slip, and the turbulent incompressible simulation uses the
$k$--$\omega$ shear-stress-transport (SST) closure with Reynolds numbers
between $10^5$ and $10^6$;
velocity fields are recorded at $T=0.5$ on 11,388 nodes~\cite{FriedenEtAl2026}. The
Mat\'ern representations use four physical lengths and Wendland uses five,
with weights prescribed by \cref{eq:power-law-kernel}. In source units,
\begin{align*}
\boldsymbol\ell_M&=(0.0003186381,\,0.0015278941,\,0.0073263699,\,0.0351305084),\\
\boldsymbol\ell_W&=(0.0008883565,\,0.0032616434,\,0.0119752803,\,0.0439678161,\,0.1614299465).
\end{align*}
Each mixture includes all the listed lengths. Over separations in
$[0.0041304,0.0179020]$, the local logarithmic slope of the kernel covariance
increment ranges from $0.627$ to $0.706$ for Mat\'ern and $0.594$ to $0.739$
for Wendland. These values describe the kernel, rather than the statistics
of a particular reconstructed velocity field.

The cavity, pooled Taylor--Green, and species datasets contain 10,200,
40,800, and 10,250 velocity fields, respectively. From each dataset, we select
five fields for testing with random seed 2032. Their indices are
3145, 3880, 5269, 6244, and 9958 for cavity; 12581, 15526, 21081,
24979, and 39840 for Taylor--Green; and 3160, 3899, 5294, 6275, and 10007 for
species. Active-node maps, geometry, and all representation parameters are
fixed across codecs. To select the residual-adaptive landmarks, we use five
additional velocity fields selected with seed 271828, none used for testing:
$(60,6031,7918,9425,9618)$ for cavity,
$(241,24130,31679,37706,38486)$ for Taylor--Green, and
$(60,6061,7957,9471,9665)$ for species. We selected landmark counts, support factors, and the condition target using pilot
experiments on the same dataset families, then fixed them for the final
comparisons. Thus the evaluation measures performance on a fixed test set;
it does not independently validate parameter selection. Ties in landmark
selection are resolved by the smallest node index, and base landmarks are
excluded from residual selection.

\subsection{Representation parameters}
\label{sec:representation-parameters}

We select the first landmarks by farthest-point sampling: each new
center is the node farthest from those already selected. For the second
set, we evaluate the residual in \eqref{eq:nystrom-level-one} on the
training velocity fields and assign each node the score
$e_i=\sum_t\|r_{\rm first}^{(t)}(x_i)\|_2^2$.
The first residual landmark maximizes $e_i$; each subsequent landmark
maximizes $\sqrt{e_i}\,\operatorname{dist}(x_i,Z_2)$ among unused nodes
outside $Z_1$. The distance is to the nearest selected residual landmark.
This combines residual magnitude with spatial coverage. Both sets are
fixed for all test velocity fields and compressors.

For cavity and Taylor--Green, we choose the global Mat\'ern shape using
the first velocity-component diagonal block $K_{11}$ of the kernel matrix.
With $\widehat\kappa_1$ denoting the estimated $1$-norm condition number,
the target $\kappa_\star=10^{12}$ and tolerance $\tau_\kappa=0.05$ give
\begin{equation}
 \left|\log_{10}\widehat\kappa_1(K_{11})-
 \log_{10}\kappa_\star\right|\leq\tau_\kappa.
 \label{eq:global-condition-target}
\end{equation}
Each \nyst{} level uses the same target and tolerance for its landmark
Gram matrix $W_j$. We search in $\log_{10}\varepsilon$, using at most 24
bracketing steps and 32 safeguarded secant--bisection steps. The \nyst{}
inverse-norm estimate uses eight Hager--Higham iterations~\cite{Higham1988}. Global coefficient
scaling uses at most 40 power iterations from a normalized all-ones vector and
stops when successive Rayleigh quotients differ by at most
$10^{-10}\max(|\lambda|,1)$, where $\lambda$ is the current Rayleigh quotient.

Species uses fixed physical lengths.
Its global first-block condition estimate is approximately $3.69\times10^6$.
The same lengths are used at every landmark count and level.

The landmark-count comparisons use
\[
\begin{array}{ll}
\text{Cavity:}& M\in\{8,16,32,64,128,256\},\\
\text{Taylor--Green:}& M\in\{8,16,32,64,128,256,512,4096\},\\
\text{Species:}& M\in\{8,16,32,64,128,256,512,1024\}.
\end{array}
\]
Here $M$ is the number of landmarks requested for each disjoint set. Two-level
\nyst{} uses $(M_1,M_2)=(M,M)$ except at the largest Taylor--Green count:
$M=4096$ gives $(M_1,M_2)=(4096,3192)$, exhausting the 7,288 active nodes.
The hybrid varies its base
landmark count and retains a full-node correction.
The individual-compressor comparisons in \cref{sec:codec-appendix} use 128
landmarks per Mat\'ern level for cavity and 256 for the other flows.
The retained matrix ranks are set by \cref{eq:nystrom-rank-threshold}
and need not equal either $M$ or $dM$. A two-Wendland-level hybrid is included only as an
ablation in \cref{sec:guarantees-ablations}.

The cavity and Taylor--Green compactly supported kernels use $\beta=6$
and the radial functions in \cref{eq:wendland-c4,eq:periodic-wendland-profile}, respectively. For
multilevel Wendland, the nested center counts are approximately
$N/2^{2d}$, $N/2^d$, and $N$ from coarse to fine. The species center hierarchy has
178, 1,424, and 11,388 centers, with the same five-length mixture at every level. The
frame uses the same center densities in fine-to-coarse order and the
simultaneous coefficient fit of \cref{sec:kernel-frame}: its species blocks have 11,388,
1,424, and 178 centers, all using the same physical kernel. Its widest support
couples approximately $90\%$ of node pairs. We therefore use dense component-block Cholesky for Species.
For cavity, the multilevel center counts are $(598,2392,9566)$, with
coarse-to-fine radii approximately $(0.240370,0.115470,0.0666665)$.
For Taylor--Green they are $(456,1822,7288)$, with radii
$(1.523627,0.826540,0.477202)$. The frame reverses these center counts;
its fine-to-coarse radii are $(0.0666665,0.133333,0.266666)$ for cavity
and $(0.477202,0.954403,1.908806)$ for Taylor--Green. The reported radii
are rounded; the spacing rules in \cref{sec:methods} define their values.

\subsection{Compression ratio}

For each velocity field, the compression ratio is the number of bytes in the
original double-precision velocity array divided by the number of bytes
returned by the compressor. For kernel methods, compression acts on the
coefficient vector; for direct compression, it acts on the nodal velocity
array. Geometry, landmark and center indices, and kernel parameters are fixed
for a dataset and are not included in the compressed size of each field. Separate frame streams include
24 bytes for their three lengths.

\subsection{Lossy compressors and coefficient tolerances}
\label{sec:coefficient-tolerances}

Every representation is paired with ZFP, PFPL, SZ3, and MGARD. We compress
the coefficient blocks as one-dimensional arrays, without reshaping them
onto an artificial multidimensional grid or applying compressor-specific
normalization. Direct compression uses the
original nodal layout. ZFP, PFPL, and SZ3 are checked against their pointwise
error bounds, and MGARD against its grid-normalized Euclidean/RMS criterion. A
compressed field is included only after independent checks of the coefficient
error bound and the reconstructed-field error.

The absolute nodal error allowances $\tau$ are
$\{10^{-3},10^{-2},10^{-1},1,\allowbreak 10,30,100,300,\allowbreak 10^3,3\!\times\!10^3,
10^4,3\!\times\!10^4\}$ for cavity,
$\{1,10,50,150,500,1200,2100,5000,10^4,2\!\times\!10^4,
4\!\times\!10^4,8\!\times\!10^4,1.6\!\times\!10^5\}$ for Taylor--Green,
and $\{1,10,100,200,300,400,10^3,3\!\times\!10^3,10^4,3\!\times\!10^4,
10^5,3\!\times\!10^5,10^6\}$ for species. If $s_D$ is an established upper
bound on the block evaluation matrix norm, ZFP, PFPL, and SZ3 receive
the absolute scalar bound $\tau/(s_D\sqrt{n_c})$, where $n_c$ is the number of coefficients;
MGARD receives the corresponding grid-normalized Euclidean/RMS tolerance. Every
coefficient error is checked in the norm controlled by its compressor.
We also check the resulting nodal error after evaluation.

Two-level \nyst{} stores both blocks as one vector and uses
$\varepsilon=\tau/\sqrt{m_1+m_2}$, as in \eqref{eq:nystrom-budget-bound}.
The hybrid and multilevel Wendland representations use the common
tolerance $\tau/\sqrt{Jm}$ from \eqref{eq:concatenated-block-bound}, with
$J=2$ and $3$, respectively. The species frame also uses $J=3$ with one
combined array. The Taylor--Green frame uses $\tau/\sqrt{3m}$ for each of
its three separate streams. The cavity frame allocates the allowance
among blocks using \eqref{eq:codec-budget-allocation}. Global interpolation
uses the power-iteration scaling described in \cref{sec:global-matern};
its applicable bound includes $b_{\rm global}$ from
\eqref{eq:global-valid-bound}.

For each shared compressor setting, the plotted compression ratio is the arithmetic
mean of the five per-field compression ratios. Relative $\ell_2$ and maximum
pointwise errors use the worst velocity field, so every reported setting satisfies
the criterion on all five. The numerator is always the number of bytes in the
double-precision source velocity array, and the denominator is the compressed size of
the corresponding nodal array or coefficient blocks.
Main curves pool all compressor results that satisfy their error bounds and remove a point $(r,e)$ whenever
another point has compression $r'\ge r$ and error $e'\le e$, with at least one
strict inequality. Exact duplicates are removed, and the remaining Pareto
points are sorted by compression ratio. Compressor-specific curves apply the same
rule without pooling compressors.

\subsection{Device-resident compression experiment}

We repeat the five kernel representations, including the landmark-count
comparisons, with CUDA ZFP,
PFPL-CUDA, cuSZ, and MGARD-X. JAX and CuPy exchange arrays through DLPack,
and the compressor libraries accept device pointers for fields,
coefficients, reconstructions, and compressed streams. Only compressed byte
counts return to the host during compression. For species, we also compute errors
and structure statistics on the GPU, downloading only scalar and binned
summaries. The Taylor--Green Wendland, frame, and hybrid experiments likewise
compute field errors on the device. Cavity and the Mat\'ern-only Taylor--Green
experiments download reconstructed arrays for subsequent error calculations;
their coefficient arrays and compressed streams remain on the device during
compression and decompression.

CUDA ZFP uses fixed rate and CPU ZFP uses fixed accuracy. cuSZ and SZ3 have
related predictive constructions with different implementations and
bitstreams. MGARD supports both CPUs and GPUs~\cite{GongEtAl2024}; here we
compare MGARD-CPU with the CUDA backend of MGARD-X, rather than running
MGARD-X on both devices. PFPL uses the same absolute-error mode on both
devices and is designed to produce identical compressed files
\cite{FallinEtAl2025}. We check every cuSZ result against
its requested absolute bound on the device. Arrays below the implementation's size threshold are stored raw. Larger
arrays also revert to raw storage if their maximum error exceeds the
on-device acceptance threshold of $1.1$ times the requested bound.
Raw storage has a coefficient-array ratio of $1$; the reported field-based
ratio still uses the original velocity-array bytes. An accepted cuSZ
reconstruction can therefore exceed its requested coefficient tolerance.
Applying \cref{thm:codec-bound} to such a result requires the acceptance
threshold $1.1\varepsilon$, or a measured coefficient-error bound, in place
of the requested $\varepsilon$. The CPU bound-verification results in
\cref{sec:bound-verification} do not use this relaxed cuSZ threshold.

\subsection{Numerical evaluation of the reported quantities}

Field error, compression ratio, divergence change, periodicity mismatch,
the structure function, and its fitted exponent are defined in
\cref{sec:results-quantities}. \Cref{eq:results-rate-error} uses the unweighted $\ell_2$ norm of the stored vector. For every constrained kernel reconstruction, divergence
violation is identically zero because
$\operatorname{range}(D)\subset\ker(\nabla\cdot)$ for every coefficient
vector. For the unconstrained point representation, the operator
$\mathcal D_h$ in \cref{eq:divergence-measure} uses the cubic PHS $r^3$,
total-degree-four polynomials, and $2\binom{4+d}{d}$ stencil nodes (30 in two
dimensions and 70 in three).
A pivoted QR selects an initial stencil subset on which the polynomial
evaluation matrix has full column rank. Convex-hull vertices are excluded
as evaluation sites but may enter neighboring stencils. For the layered
three-dimensional species nodes, the thin coordinate is scaled to the median
in-layer spacing and the corresponding velocity component is scaled
identically, leaving divergence unchanged.
Applying $\mathcal D_h$ to a divergence-free kernel field produces stencil
error. We use this calculation to check nodal consistency and use the
analytic divergence in the tradeoff plots, displaying exact zeros at a
labeled floor.

For species transport, a seed-10 subsample of 1,500 nodes defines all pairs,
which are accumulated in 700 uniformly spaced separation bins. With
$S_2(r_k)=\langle\|u(x_i)-u(x_j)\|_2^2\rangle_{\|x_i-x_j\|\in k}$, we report
relative curve error over bins $k=10,\ldots,47$,
numbered from zero. If $r_{\min}$ and $r_{\max}$ are the smallest and largest
sampled pair separations and $\Delta r=(r_{\max}-r_{\min})/700$, the fitted
bin centers run from $r_{\min}+10.5\Delta r$ to
$r_{\min}+47.5\Delta r$. The exponent $p$ is obtained by nonlinear
least squares in $S_2(r)=Cr^p$ over this fixed interval. We keep this interval fixed when comparing source and reconstructed fields.
The sensitivity study in \cref{sec:structure-window} assesses the effect
of changing it.

\subsection{Implementation and reproducibility checks}

The formulas in \cref{sec:methods} place velocity components at each node
next to one another. The implementation groups values by component;
applying the same permutation to matrices and vectors preserves the
Euclidean norms and singular values used in the error estimates.
\label{sec:implementation}

Kernel assembly, landmark selection, dense block factorization, sparse direct
solution, residual updates, and evaluation use double precision in JAX on
CUDA GPUs. The \texttt{POCON} routine in the Linear Algebra Package
(LAPACK) on the host evaluates the condition estimate used for the
global shape search. File input and output, periodic-node identification, Karush--Kuhn--Tucker (KKT) system assembly,
CPU codec calls, and transfers required by those calls remain on the host.
The device-resident codecs eliminate bulk transfers during compression and
decompression. Species error statistics are also computed on the device;
the cavity and Mat\'ern-only Taylor--Green studies transfer reconstructed arrays
for reporting. Taylor--Green Wendland, frame, and hybrid errors are computed
on the device. The
global solver uses coupled block Cholesky with nested Schur complements; the
sparse methods use direct factorization.
For cavity and Taylor--Green, multilevel Wendland uses JAX sparse \texttt{spsolve} with tolerance
$10^{-10}$. The frame KKT solve uses tolerance $10^{-12}$ with fill-reducing reordering,
and each uncompressed frame reconstruction must have relative residual below
$10^{-9}$. For each method and landmark count, we check that every prescribed
compressor setting was evaluated exactly once for each field, with finite
errors and consistent byte counts. We also check condition
estimates, coefficient tolerances, and independently recomputed field errors.
Before compression, we evaluate each fitted representation: interpolatory
methods must recover the data to solver tolerance, and the \nyst{} projection
error is measured explicitly.
Species uses dense block solves for its broad Wendland mixture and the
minimum-norm formula \eqref{eq:frame-shared-kernel} for its frame. Frame coefficients for
cavity and Taylor--Green use deterministic Morton ordering within each block.
For species, the frame blocks are concatenated from fine to coarse, with
velocity components grouped within each block. The center order begins
with the farthest-point sequence defining the intermediate level, followed
by the remaining node indices in ascending order. Coarser levels use
prefixes of this sequence; no Morton permutation is applied.

Source code and synthetic examples for the five representations are available at
\url{https://github.com/VarShankar/property-preserving-compression}.
The repository depends on KernelPack-JAX; the paper datasets and archived
outputs are not included. The software environment for the reported experiments comprised
Python 3.12.13, NumPy 2.5.3, SciPy 1.18.1, JAX/JAXlib 0.11.2, zfpy 1.0.1, and
pysz 1.0.3. The MGARD build uses 64-bit storage for Huffman outlier indices.
Experiments ran on an AMD Ryzen 7
9800X3D host with 64 gibibytes of memory and an NVIDIA RTX 5090
(32 gibibytes, driver 610.62).

%% file: bound_details.tex
\section{Additional nodal error comparisons}
\label{sec:bound-details}

To interpret the bound comparisons, we distinguish evaluation-norm
estimation, finite-precision fitting errors, and violations of the requested
coefficient tolerance. For global Mat\'ern on cavity and Taylor--Green, the evaluation-norm bound in
\cref{eq:global-valid-bound} is computed from the absolute row sums of
the symmetric kernel matrix, assembled in blocks on the GPU.
We reconstruct the power-iteration scaling using the same initial vector,
iteration limit, and stopping tolerance. Different floating-point summation
orders can change this recomputed estimate; the original estimate was not saved.
Each test combines a velocity field, representation and landmark count, compressor, and
tolerance allocation. Repeated tests on one velocity field are statistically
dependent.
Five global-Mat\'ern cavity tests at $\tau=10^{-3}$ exceeded MGARD's
requested RMS coefficient error. We exclude these tests from the theorem comparison because they violate
its coefficient-error assumption, even though their nodal errors are below
$\tau$.
The computed factors in \eqref{eq:global-valid-bound} are
$b_{\rm global}=1.087$ for cavity and $1.277$ for Taylor--Green.
For the coefficient tolerances used, the corresponding allowances are
$b_{\rm global}\tau$. Among settings meeting their coefficient bounds,
the largest observed fractions of these allowances are $0.037$ and
$0.015$, respectively. These norm calculations use floating-point arithmetic.

For the cavity frame results, the stored total field error and uncompressed interpolation
residual give upper and lower bounds on the nodal change due to compression through
\eqref{eq:compression-error-bracket}. All resulting upper bounds remain below the allocated field
allowance. For entries lacking a saved coefficient-error measurement, we
repeat compression and decompression of the saved coefficient vectors with
the original compressor settings and within-level ordering. Every repeated
compressed byte count agrees with the original table, and every measured
coefficient error satisfies the requested bound. 
For Taylor--Green, the periodic Wendland, frame, and hybrid experiments
measure the nodal change due to coefficient compression directly. Across
settings satisfying the coefficient-error bound, the largest fractions of
the allocated nodal allowance are
$0.051$ for multilevel Wendland, $0.055$ for the Morton-ordered blockwise
frame, and $0.211$ over the hybrid landmark counts. These fractions measure the bounds' conservatism for the observed
coefficient perturbations. One frame ZFP test at $\tau=40000$ exceeds its
requested maximum coefficient error by about $10\%$, although its nodal
error remains below $0.008\tau$. We exclude that setting from the
curves and theorem comparison that require the coefficient error bound to hold.

For the cavity and Taylor--Green two-level \nyst{} bases, we separately repeat the fits and test
the orthogonality and total-error identity in \eqref{eq:nystrom-nodal-bound} at every landmark count and
velocity field. The largest Frobenius norms of $U_j^TU_j-I$ and $U_1^TU_2$ are
$5.8\times10^{-12}$ and $2.0\times10^{-13}$, respectively.
We perturb the coefficient vector in the direction
$v_k=\sin(k+0.37)$, rescaled so that
$\|\delta c\|_2/\|u_X\|_2$ equals $10^{-8}$, $10^{-4}$, or $10^{-1}$.
The resulting nodal perturbation satisfies
$\big|\|U\delta c\|_2/\|\delta c\|_2-1\big|<1.2\times10^{-9}$,
including evaluation roundoff. The two sides of the squared total-error identity differ
by less than $4.1\times10^{-15}\|u_X\|_2^2$.
Relative to the squared total error itself, the largest discrepancy is
$3.2\times10^{-7}$, attained when both the projection error and the
perturbation are small. These deterministic perturbations check the evaluation matrix identities.
\Cref{tab:nystrom-bound-by-rank} reports the separate compression experiments.

\begin{table*}[tbp]
\centering
\caption{Maximum observed fraction $\theta=\|A\delta c\|_2/\tau$
from \eqref{eq:bound-effectivity} for each
landmark count. The columns give the number of landmarks requested at each level for
two-level \nyst{}; the Taylor--Green $M=4096$ case uses 4,096 and 3,192
landmarks. The hybrid varies the Mat\'ern landmark count while retaining
the full-node Wendland correction. The column headings are landmark counts, not retained matrix ranks.
Maxima pool five fields and CPU compressor settings satisfying their coefficient
bounds. Dashes denote untested counts; one is the predicted upper bound.}
\label{tab:nystrom-bound-by-rank}
\small
\setlength{\tabcolsep}{3.2pt}
\begin{tabular}{llrrrrrrrrr}
\toprule
Method & Data & 8 & 16 & 32 & 64 & 128 & 256 & 512 & 1024 & 4096\\
\midrule
Two-level \nyst{} & Cavity
 & .680 & .635 & .621 & .615 & .606 & .589 & --- & --- & ---\\
 & Taylor--Green
 & .567 & .355 & .258 & .526 & .168 & .520 & .169 & --- & .299\\
 & Species
 & .635 & .635 & .614 & .604 & .597 & .587 & .584 & .581 & ---\\
\addlinespace
\nyst{} + Wendland & Cavity
 & .058 & .058 & .059 & .058 & .067 & .081 & --- & --- & ---\\
Hybrid & Taylor--Green
 & .077 & .081 & .039 & .055 & .077 & .101 & .158 & --- & .211\\
 & Species
 & .017 & .020 & .024 & .034 & .046 & .063 & .087 & .120 & ---\\
\bottomrule
\end{tabular}
\end{table*}

\subsection{Sensitivity of the structure-function exponent}
\label{sec:structure-window}

We compare the compression-induced exponent changes with the sensitivity
of the exponent estimate to the fitting interval. We repeat the original-field
fit on four neighboring intervals, keeping the
same point subsample, separation bins, and nonlinear least-squares objective.
We shift, shorten, and extend the original interval.
The mean exponent over the five velocity fields changes as follows:
\begin{center}
\begin{tabular}{lrr}
\toprule
Interval & Separation range & Mean exponent\\
\midrule
Shifted toward smaller separations & $[0.00227,0.01604]$ & 0.762\\
Shortened & $[0.00413,0.01418]$ & 0.739\\
Original & $[0.00413,0.01790]$ & 0.685\\
Shifted toward larger separations & $[0.00599,0.01976]$ & 0.606\\
Extended & $[0.00413,0.02162]$ & 0.507\\
\bottomrule
\end{tabular}
\end{center}
Separations use the source coordinate units. These changes are much larger
than the compression-induced shifts reported on the original interval.
The fitted exponent is therefore sensitive to the interval and cannot
identify a robust inertial-range law. The reported compression-induced
changes apply to the fixed interval used in the main comparisons.